\documentclass[11pt,twoside]{amsart}
\usepackage[latin1]{inputenc}
\usepackage[OT1]{fontenc}
\usepackage{amsmath}
\usepackage{amsthm}
\usepackage{amssymb}
\usepackage[all,cmtip]{xy}
\usepackage[usenames,dvipsnames]{xcolor}

\definecolor{prussian}{RGB}{0,49,83}
\definecolor{hunter}{RGB}{53,94,59}

\usepackage{hyperref}
\hypersetup{
    colorlinks,
    linkcolor={red!50!black},
    citecolor={blue!50!black},
    urlcolor={blue!80!black}
}
\usepackage{aliascnt} 
\newcommand{\hyref}[1]{\hyperref[#1]{\ref{#1}}}

\usepackage{bbm}
\usepackage{amscd}
\usepackage{enumitem}
\usepackage{graphicx}
\usepackage{etoolbox}
\usepackage{a4wide}
\usepackage{graphicx}

\usepackage{scalerel} 
\newcommand{\scalemath}[2]{\hstretch{#1}{\vstretch{#1}{#2}}}

\newcommand{\inv}{^{-1}}
\newcommand{\orth}{^{\perp}}
\newcommand{\lorth}{{}\orth}
\newcommand{\dual}{^\vee}

\DeclareMathOperator{\Spec}{Spec}

\DeclareMathOperator{\id}{id}

\DeclareMathOperator{\Hom}{Hom}

\DeclareMathOperator{\sHom}{\mathcal{H}\textit{om}}

\DeclareMathOperator{\End}{End}

\DeclareMathOperator{\Coh}{Coh}
\DeclareMathOperator{\QCoh}{QCoh}

\DeclareMathOperator{\Ho}{H}

\DeclareMathOperator{\Hilb}{Hilb}

\DeclareMathOperator{\codim}{codim}

\DeclareMathOperator{\rank}{rank}
\DeclareMathOperator{\supp}{\mathsf{supp}}

\DeclareMathOperator{\triv}{\mathsf{triv}}

\DeclareMathOperator{\ord}{ord}

\DeclareMathOperator{\Bl}{Bl}

\DeclareMathOperator{\cone}{cone}

\DeclareMathOperator{\CH}{\mathcal{H}}
\DeclareMathOperator{\CY}{\mathrm{CY}}
\DeclareMathOperator{\WC}{\mathrm{WCN}}
\DeclareMathOperator{\Dperf}{D^{\mathrm{perf}}}
\DeclareMathOperator{\DperfG}{D^{\mathrm{perf}}_{\mathit{G}}}

\DeclareMathOperator{\disc}{\mathrm{disc}}

\newcommand{\Db}{{\rm D}^{\rm b}}
\newcommand{\D}{{\rm D}}

\DeclareMathOperator{\Aut}{\mathrm{Aut}}

\DeclareMathOperator{\Pic}{\mathrm {Pic}}

\newcommand{\into}{\hookrightarrow}

\newcommand{\mor}[1][]{\xrightarrow{#1}}
\newcommand{\isomor}{\mor[\sim]}  

\newcommand{\ka}{{\mathcal A}}

\newcommand{\kc}{{\mathcal C}}
\newcommand{\kd}{{\mathcal D}}
\newcommand{\ke}{{\mathcal E}}

\newcommand{\kh}{{\mathcal H}}

\newcommand{\cI}{{\mathcal I}}

\newcommand{\kl}{{\mathcal L}}

\newcommand{\ko}{{\mathcal O}}

\newcommand{\kt}{{\mathcal T}}
\newcommand{\ku}{{\mathcal U}}

\newcommand{\wPhi}{{\widetilde \Phi}}

\newcommand{\pt}{\mathsf{pt}}

\newcommand{\IA}{\mathbb{A}}
\newcommand{\IC}{\mathbb{C}}

\newcommand{\IN}{\mathbb{N}}
\newcommand{\IP}{\mathbb{P}}

\newcommand{\IZ}{\mathbb{Z}}

\DeclareMathOperator{\Ind}{\mathsf{Ind}}
\DeclareMathOperator{\Res}{\mathsf{Res}}

\DeclareMathOperator{\FM}{\mathsf{FM}}
\DeclareMathOperator{\TT}{\mathsf{T}\!}
\DeclareMathOperator{\MM}{\mathsf{M}}

\DeclareMathOperator{\Fix}{\mathsf{Fix}}
\newcommand{\Serre}{\mathsf{S}}

\newcommand{\sym}{\mathfrak S}

\newcommand{\cA}{\mathcal A}
\newcommand{\cB}{\mathcal B}

\newcommand{\cF}{\mathcal F}

\newcommand{\cE}{\mathcal E}

\newcommand{\cX}{\mathcal X}
\newcommand{\cY}{\mathcal Y}
\newcommand{\cZ}{\mathcal Z}

\newcommand{\cC}{\mathcal C}
\newcommand{\cD}{\mathcal D}

\newcommand{\cP}{\mathcal P}

\newcommand{\cT}{\mathcal T}
\newcommand{\cL}{\mathcal L}

\newcommand{\cS}{\mathcal S}

\newcommand{\cV}{\mathcal V}
\newcommand{\cW}{\mathcal W}

\newcommand{\fm}{\mathfrak m}

\newcommand{\fn}{\mathfrak n}
\newcommand{\reg}{\mathcal O}

\newcommand{\wcT}{{\widetilde{\mathcal T}} }

\newcommand{\wwY}{{\scalemath{0.9}{\widetilde Y}}}
\newcommand{\wY}{{\widetilde Y}}
\newcommand{\wwX}{{\scalemath{0.9}{\widetilde X}}}
\newcommand{\wX}{{\widetilde X}}

\renewcommand{\theta}{\vartheta}
\renewcommand{\rho}{\varrho}
\renewcommand{\phi}{\varphi}
\renewcommand{\_}{\underline{\,\,\,\,}}

\newcommand{\red}{{\mathrm{red}}}

\newcommand{\sod}[1]{\langle #1\rangle}
\newcommand{\Sod}[1]{\big\langle #1\big\rangle}
\newcommand{\gen}[1]{\langle #1\rangle}

\newcommand{\perf}{\text{perf}}

\usepackage{mathtools}

\newcommand{\coloneqq}{\mathrel{\mathop:}=}
\newcommand{\eqqcolon}{=\mathrel{\mathop:}}

\newtheorem{theorem}{Theorem}[section]

  \newaliascnt{proposition}{theorem}
  \newtheorem{proposition}[proposition]{Proposition}
  \aliascntresetthe{proposition}

  \newaliascnt{lemma}{theorem}
  \newtheorem{lemma}[lemma]{Lemma}
  \aliascntresetthe{lemma}

  \newaliascnt{corollary}{theorem}
  \newtheorem{corollary}[corollary]{Corollary}
  \aliascntresetthe{corollary}

\newtheorem*{propositionA}{Proposition A}
\newtheorem*{theoremB}{Theorem B}
\newtheorem*{theoremC}{Theorem C}

\theoremstyle{definition}

  \newaliascnt{definition}{theorem}
  \newtheorem{definition}[definition]{Definition}
  \aliascntresetthe{definition}

  \newaliascnt{remark}{theorem}
  \newtheorem{remark}[remark]{Remark}
  \aliascntresetthe{remark}

  \newaliascnt{condition}{theorem}
  \newtheorem{condition}[condition]{Condition}
  \aliascntresetthe{condition}

  \newaliascnt{question}{theorem}
  \newtheorem{question}[question]{Question}
  \aliascntresetthe{question}

  \newaliascnt{example}{theorem}
  
  \aliascntresetthe{example}

\begin{document}

\title[Derived categories of resolutions]{Derived categories of resolutions of cyclic quotient singularities}
\author[A.\ Krug]{Andreas Krug}
\author[D.\ Ploog]{David Ploog}
\author[P.\ Sosna]{Pawel Sosna}

\begin{abstract}
For a cyclic group $G$ acting on a smooth variety $X$ with only one character occurring in the $G$-equivariant decomposition of the normal bundle of the fixed point locus, we study the derived categories of the orbifold $[X/G]$ and the blow-up resolution $\wY\to X/G$. 
  
Some results generalise known facts about $X=\IA^n$ with diagonal $G$-action, while other results are new also in this basic case.
In particular, if the codimension of the fixed point locus equals $|G|$, we study the induced tensor products under the equivalence $\Db(\wY)\cong \Db([X/G])$ and give a 'flop-flop=twist' type formula.
We also introduce candidates for general constructions of categorical crepant resolutions inside the derived category of a given geometric resolution of singularities and test these candidates on cyclic quotient singularities.    
\end{abstract}

\begingroup\renewcommand\thefootnote{}%
%
\addtocounter{footnote}{-2}\endgroup

\maketitle

\noindent
\setcounter{tocdepth}{1}
{\centering\parbox{0.7\textwidth}{\tableofcontents}\par}

\addtocontents{toc}{\protect\setcounter{tocdepth}{-1}}     
\section*{Introduction}
\addtocontents{toc}{\protect\setcounter{tocdepth}{1}}      

\noindent
For geometric, homological and other reasons, it has become commonplace to study the bounded derived category of a variety. 
One of the many intriguing aspects are connections, some of them conjectured, some of them proven, to birational geometry. 

One expected phenomenon concerns a birational correspondence
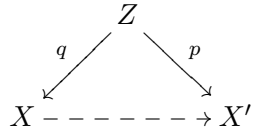
\[ \xymatrix{
                  & Z\ar_q[dl] \ar^p[dr] &  \\ 
  X \ar@{-->}[rr] &                      & X'
} \]
of smooth varieties. Then we should have:
\begin{itemize}
  \item\label{conj1} A fully faithful embedding $\Db(X)\hookrightarrow \Db(X')$, if $q^*K_X\le p^* K_{X'}$.
  \item\label{conj2} A fully faithful embedding $\Db(X')\hookrightarrow \Db(X)$, if $q^*K_X\ge p^* K_{X'}$.
  \item An equivalence $\Db(X')\cong \Db(X)$, in the \emph{flop} case $q^*K_X= p^* K_{X'}$.
\end{itemize}
This is proven in many instances; see \cite{Bondal-Orlov2}, \cite{Bridgelandflops}, \cite{KawDK}, \cite{NamikawaMukaiflops}.

Another very interesting aspect of derived categories is their occurrence in the context of the McKay correspondence. Here, one of the key expectations is that the derived category of a crepant resolution $\wY\to X/G$ of a Gorenstein quotient variety is derived equivalent to the corresponding quotient orbifold:
 $\Db(\wY) \cong \Db([X/G]) = \Db_G(X)$.     
In \cite{BKR}, this expectation is proven in many cases under the additional assumption that $\wY\cong \Hilb^{G}(X)$ is the fine moduli space of $G$-clusters on $X$.
It is enlightening to view the derived McKay correspondence as an orbifold version of the conjecture on derived categories under birational correspondences described above; for more information on this point of view, see \cite[Sect.~2]{Kaw16}, where the conjecture is called the \textit{DK-Hypothesis}. Indeed, if we denote the universal family of $G$-clusters by $\cZ\subset \wY\times X$, we have the following diagram of birational morphisms of orbifolds
\begin{equation}
\begin{aligned}\label{orbiflop}
\xymatrix@!C=2em{
                               & [\cZ/G]\ar_q[dl] \ar^p[dr] &  \\ 
\wY \ar@{-->}[rr] \ar_\rho[dr] & & [X/G] \ar^\pi[dl] \\
                               & X/G
\,.}
\end{aligned}
\end{equation}
Since the pullback of the canonical sheaf of $X/G$ under $\pi$ is the canonical sheaf of $[X/G]$, the condition that $\rho$ is a crepant resolution amounts to saying that \eqref{orbiflop} is a flop of orbifolds. 

In many situations, a crepant resolution of $X/G$ does not exist. However, given a resolution $\rho\colon \wY\to X$, the DK-Hypothesis still predicts the behavior of the categories $\Db(\wY)$ and $\Db_G(X)$ if $\rho^*K_{X/G}\ge K_{\wwY}$ or $\rho^*K_{X/G}\le K_{\wwY}$.
Another related idea is that, even though a crepant resolution does not exist in general, there should always be a \textit{categorical crepant resolution} of $\Db(X/G)$; see \cite{KuzLefschetz}.
The hope is to find such a categorical resolution as an admissible subcategory of the derived category $\Db(\wY)$ of a geometric resolution.

Besides dimensions 2 and 3, one of the most studied testing grounds for the above, and related, ideas is the isolated quotient singularity $\IA^n/\mu_m$. Here, the cyclic group $\mu_m$ of order $m$ acts on the affine space 
by multiplication with a primitive $m$-th root of unity $\zeta$. In this paper, we consider the following straight-forward generalisation of this set-up. Namely, let $X$ be a quasi-projective smooth complex variety acted upon by the finite cyclic group $\mu_m$. We assume that only $1$ and $\mu_m$ occur as the isotropy groups of the action and write $S\coloneqq\Fix(\mu_m)\subset X$ for the fixed point locus. Fix a generator $g$ of $\mu_m$ and assume that $g$ acts on the normal bundle $N\coloneqq N_{S/X}$ by multiplication with some fixed primitive $m$-th root of unity $\zeta$. Then the blow-up $\wY\to X/\mu_m$ with center $S$ is a resolution of singularities; see \autoref{sec:setup} for further details. There are four particular cases we have in mind:   

\begin{enumerate}[label=(\alph*)]
\item\label{basiccase}
$X=\IA^n$ with the diagonal action of any $\mu_m$.
\item
$X=Z^2$, where $Z$ is a smooth projective variety of  arbitrary dimension, and $\mu_2=\sym_2$ acts by permuting the factors. Then $\wY\cong Z^{[2]}$, the Hilbert scheme of two points.
\item
$X$ is an abelian variety, $\mu_2$ acts by $\pm 1$. In this case, $\wY$ is known as the \textit{Kummer resolution}.
\item $X \to Y=X/\mu_m$ is a cyclic covering of a smooth variety $Y$, branched over a divisor. Here, $n=1$ and $\wX=X$, $\wY=Y$. This case has been studied in \cite{KuzPerrycyclic}.
\end{enumerate} 

First, we prove the following result in \autoref{sub:clusters}. This is probably well-known to experts, but we could not find it in the literature. Write $G \coloneqq \mu_m$.

\begin{propositionA}[=\ \autoref{prop:Hilbert-resolution}]
The resolution obtained by blowing up the fixed point locus in $X/G$ is isomorphic to the $G$-Hilbert scheme: 
$\wY \cong \Hilb^{G}(X)$.
\end{propositionA}

We set $n \coloneqq \codim(S\hookrightarrow X)$ and find the following dichotomy, in accordance with the DK-Hypothesis. We keep the notation from diagram \eqref{orbiflop}. In particular, for $n=m$, we obtain new instances of BKR-style derived equivalences between orbifold and resolution.

\begin{theoremB}[=\ \autoref{thm:mainprecise}]\noindent\label{thm:main}
\begin{enumerate}
 \item The functor $\Phi \coloneqq p_*q^*\colon \Db(\wY)\to \Db_G(X)$ is fully faithful for $m\ge n$ and an equivalence for $m=n$. For $m>n$, there is a semi-orthogonal decomposition 
of $\Db_G(X)$ consisting of $\Phi(\Db(\wY))$ and $m-n$ pieces equivalent to $\Db(S)$. 
\item The functor $\Psi \coloneqq q_*p^*\colon\Db_G(X)\to \Db(\wY) $ is fully faithful for $n\ge m$ and an equivalence for $n=m$. For $n>m$, there is a semi-orthogonal decomposition 
of $\Db(\wY)$ consisting of $\Psi(\Db_G(X))$ and $n-m$ pieces equivalent to $\Db(S)$.
\end{enumerate}
\end{theoremB}

For a more exact statement with an explicit description of the embeddings of the $\Db(S)$ components into $\Db(\wY)$ and $\Db_G(X)$, see \autoref{sec:main-proof}.
In particular, for $m>n$, the push-forward $a_*\colon \Db(S)\to \Db_G(X)$ along the embedding $a\colon S\hookrightarrow X$ of the fixed point divisor is fully faithful.

In the basic affine case \ref{basiccase}, the result of the theorem is also stated in \cite[Ex.~4]{Kaw16} and there are related results in the more general toroidal case in \cite[Sect.~3]{Kaw16}. Proofs, in the basic case, are given in \cite[Sect.~4]{Abuaf-catres} for $n\ge m$ and in \cite{IUspecial} for $n=2$. If $n=1$, the quotient is already smooth and we have $\wY=X/G$ --- here the semi-orthogonal decomposition categorifies the natural decomposition of the orbifold cohomology; compare \cite{PolvdB}. The $n=1$ case is also proven in \cite[Thm.~3.3.2]{Lim}.

We study the case $m=n$, where $\Phi$ and $\Psi$ are equivalences, in more detail.  
On both sides of the equivalence, we have distinguished line bundles. The line bundle $\reg_{\widetilde Y}(Z)$ on $\wY$, corresponding to the exceptional divisor, admits an $m$-th root $\cL$. On $[X/G]$, there are twists of the trivial line bundle by the group characters $\reg_X\otimes \chi^i$. For $i=-m+1, \dots, -1,0$, we have $\Psi(\reg_X\otimes \chi^i)\cong \cL^i$.  
Furthermore, we see that the functors $\Db(S)\to \Db(\wY)$ and $\Db(S)\to \Db_G(X)$, which define fully faithful embeddings in the $n>m$ and $m<n$ cases, respectively, become spherical for $m=n$ and hence induce twist autoequivalences; see \autoref{sub:spherical} for details on spherical functors and twists.
We show that the tensor products by the distinguished line bundles correspond to the spherical twists under the equivalences $\Psi$ and $\Phi$. In particular, one part of \autoref{thm:n=m} is the following formula.

\begin{theoremC}\label{thm:tensor}
There is an isomorphism
  $\Psi^{-1}(\Psi(\_)\otimes \cL^{-1}) \cong \TT^{\;-1}_{a_*}(\_\otimes \chi^{-1})$
of autoequivalences of $\Db_G(X)$ where the inverse spherical twist $\TT^{\;-1}_{a_*}$ is defined by the exact triangle of functors
\[
   \TT^{\;-1}_{a_*} \to \id \to a_*(a^*(\_)^G) \to
\,. \]
\end{theoremC}

The tensor powers of the line bundle $\cL$ form a strong generator of $\Db(\wY)$, thus \autoref{thm:tensor}, at least theoretically, completely describes the tensor product 
\[
 \_\widehat\otimes \_ \coloneqq \Psi^{-1}(\Psi(\_)\otimes \Psi(\_))\colon \Db_G(X)\times \Db_G(X)\to \Db_G(X)
\]
induced by $\Psi$ on $\Db_G(X)$. 
There is related unpublished work on induced tensor products under the McKay correspondence in dimensions 2 and 3 by T. Abdelgadir, A. Craw, J. Karmazyn, and A. King.
In \autoref{cor:flopfloptwistcor}, we also get a formula which can be seen as a stacky instance of the 'flop-flop = twist' principle as discussed in \cite{ADMflop}.

In \autoref{sec:categorical-resolutions}, we introduce a general candidate for a weakly crepant categorical resolution (see \cite{KuzLefschetz} or \autoref{sub:catresdef} for this notion), namely the \textit{weakly crepant neighbourhood} $\WC(\rho)\subset \Db(\wY)$, inside the derived category of a given resolution $\rho\colon\wY\to Y$ of a rational Gorenstein variety $Y$. The idea is pretty simple: by Grothendieck duality, there is a canonical section $s\colon \reg_\wwY\to \reg_\rho$ of the relative dualising sheaf, and this induces a morphism of Fourier--Mukai transforms
$
  t \coloneqq \rho_*(\_\otimes s)\colon \rho_* \to \rho_!  \, .
$

Set $\rho_+ \coloneqq \cone(t)$ and $\WC(\rho) \coloneqq \ker(\rho_+)$. Then, by the very construction, we have $\rho_{*\mid \WC(\rho)}\cong \rho_{!\mid \WC(\rho)}$ which amounts to the notion of categorical weak crepancy. There is one remaining condition needed to ensure that $\WC(\rho)$ is a categorical weakly crepant resolution: whether it is actually a smooth category; this holds as soon as it is an admissible subcategory of $\Db(\wY)$ which means that its inclusion has adjoints. 
We prove that, in the Gorenstein case $m\mid n$ of our set-up of cyclic quotients, $\WC(\rho)\subset \Db(\wY)$ is an admissible subcategory; see \autoref{thm:cyclicWC}. 

In \autoref{sub:non-unicity}, we observe that there are various weakly crepant resolutions inside $\Db(\wY)$. However, a strongly crepant categorical resolution inside $\Db(\wY)$ is unique, as we show in \autoref{prop:unicity}. Our concept of weakly crepant neighbourhoods was motivated by the idea that some non-CY objects possess `CY neighbourhoods'' (a construction akin to the spherical subcategories of spherelike objects in \cite{HKP}), i.e.\ full subcategories in which they become Calabi--Yau. This relationship is explained in \autoref{sub:CY-neighbourhoods}.

In the final \autoref{sec:Kummer}, we construct Bridgeland stability conditions on Kummer threefolds as an application of our results; see \autoref{cor:stability}.

\bigskip
\noindent
\textbf{Conventions.} 
We work over the complex numbers. All functors are assumed to be derived. 
We write $\kh^i(E)$ for the $i$-th cohomology object of a complex $E\in \Db(Z)$ and $\Ho^*(E)$ for the complex $\oplus_i \Ho^i(Z,E)[-i]$. If a functor $\Phi$ has a left/right adjoint, they are denoted $\Phi^L$, $\Phi^R$.

There are a number of spaces, maps and functors repeatedly used in this text. For the convenience of the reader, we collect our notation at the very end of this article, on \autopageref{notation_recap}.

\bigskip
\noindent
\textbf{Acknowledgements.}
It is a pleasure to thank
Tarig Abdelgadir, Martin Kalck, S\"onke Rollenske and Evgeny Shinder
for comments and discussions.
We are grateful to the anonymous referee for very careful inspection.

\section{Preliminaries}\label{sec:preliminaries}

\subsection{Fourier--Mukai transforms and kernels}\label{sub:fmtransforms}
Recall that given an object $\ke$ in $\Db(Z\times Z')$, where $Z$ and $Z'$ are smooth and projective, we get an exact functor $\Db(Z)\to \Db(Z')$, $F \mapsto p_{Z'*}(\ke\otimes p_Z^* F)$. Such a functor, denoted by $\FM_\ke$, is called a \emph{Fourier--Mukai transform} (or FM transform) and $\ke$ is its kernel. See \cite{Huy} for a thorough introduction to FM transforms. For example, if $\Delta\colon Z\to Z\times Z$ is the diagonal map and $\kl$ is in $\Pic(Z)$, then $\FM_{\Delta_*\kl}(F) = F\otimes\kl$. In particular, $\FM_{\ko_\Delta}$ is the identity functor.
\smallskip

\noindent
{\textbf{Convention.}} We will write $\MM_\kl$ for the functor $\FM_{\Delta_*\kl}$. 
\smallskip

The calculus of FM transforms is, of course, not restricted to smooth and projective varieties. Note that $f_*$ maps $\Db(Z)$ to $\Db(Z')$ as soon as $f\colon Z\to Z'$ is proper. In order to be able to control the tensor product and pullbacks, one can restrict to perfect complexes. Recall that a complex of sheaves on a quasi-projective variety $Z$ is called \emph{perfect} if it is locally quasi-isomorphic to a bounded complex of locally free sheaves. The triangulated category of perfect complexes on $Z$ is denoted by $\Dperf(Z)$. It is a full subcategory of $\Db(Z)$. These two categories coincide if and only if $Z$ is smooth.  

We will sometimes take cones of morphisms between FM transforms. Of course, one needs to make sure that these cones actually exist. Luckily, if one works with FM transforms, this is not a problem, because the maps between the functors come from the underlying kernels and everything works out, even for (reasonable) schemes which are not necessarily smooth and projective; see \cite{AL}.

\subsection{Group actions and derived categories}\label{sub:equivariant} 
Let $G$ be a finite group acting on a smooth variety $X$. Recall that a $G$-equivariant coherent sheaf is a pair $(F,\lambda_g)$, where $F \in \Coh(X)$ and $\lambda_g\colon F\isomor g^*F$ are isomorphisms satisfying a cocycle condition.
The category of $G$-equivariant coherent sheaves on $X$ is denoted by $\Coh^G(X)$. It is an abelian category. The \emph{equivariant derived category}, denoted by $\Db_G(X)$, is defined as $\Db(\Coh^G(X))$, see, for example, \cite{Plo} for details.  
Recall that for every subgroup $G'\subset G$ the restriction functor $\Res\colon \Db_G(X)\to \Db_{G'}(X)$ has the induction functor $\Ind\colon \Db_{G'}(X)\to \Db_G(X)$ as a left and right adjoint (see e.g.\ \cite[Sect.~1.4]{Plo}). It is given for $F\in \Db(Z)$ by
\begin{align}\label{eq:induction-definition}
  \Ind(F) = \bigoplus_{[g]\in G'\setminus G}g^* F
\end{align}
with the $G$-linearisation given by the $G'$-linearisation of $F$ together with appropriate permutations of the summands.

If $G$ acts trivially on $X$, there is also the functor $\triv\colon\Db(X)\to \Db_{G}(X)$ which equips an object with the trivial $G$-linearisation. Its left and right adjoint is the functor $(\_)^G\colon \Db_G(X)\to \Db(X)$ of invariants.

Given an equivariant morphism $f\colon X\to X'$ between varieties endowed with $G$-actions, there are equivariant pushforward and pullback functors, see, for example, \cite[Sect.~1.3]{Plo} for details. We will sometimes write $f_*^G$ for $(\_)^G\circ f_*$. It is also well-known that the category $\Db_G(X)$ has a tensor product and the usual formulas, e.g.\ the adjunction formula, hold in the equivariant setting. 

Finally, we need to recall that a character $\kappa$ of $G$ acts on the equivariant category by twisting the linearisation isomorphisms with $\kappa$. If $F\in \Db_G(X)$, we will write $F\otimes \kappa$ for this operation. We will tacitly use that twisting by characters commutes with the equivariant pushforward and pullback functors along $G$-equivariant maps.

\subsection{Semi-orthogonal decompositions}\label{sub:semiorth}
References for the following facts are, for example, \cite{Bon} and \cite{Bondal-Orlov2}. 

Let $\kt$ be a Hom-finite triangulated category. A \emph{semi-orthogonal decomposition} of $\kt$ is a sequence of full triangulated subcategories $\ka_1,\ldots,\ka_m$ such that (a) if $A_i\in \cA_i$ and $A_j\in \cA_j$, then $\text{Hom}(A_i,A_j[l])=0$ for $i>j$ and all $l$, and (b) the $\cA_i$ generate $\kt$, that is, the smallest triangulated subcategory of $\kt$ containing all the $\cA_i$ is already $\kt$. We write $\kt=\sod{\ka_1,\ldots,\ka_m}$. If $m=2$, these conditions boil down to the existence of a functorial exact triangle $A_2\to T\to A_1$ for any object $T\in \kt$.

A subcategory $\ka$ of $\kt$ is \emph{right admissible} if the embedding functor $\iota$ has a right adjoint $\iota^R$, \emph{left admissible} if $\iota$ has a left adjoint $\iota^L$, and \emph{admissible} if it is left and right admissible. 

Given any triangulated subcategory $\ka$ of $\kt$, the full subcategory $\ka\orth\subseteq\kt$ consists of objects $T$ such that $\Hom(A,T[k])=0$ for all $A\in \ka$ and all $k\in \IZ$. If $\ka$ is right admissible, then $\kt=\sod{\ka\orth,\ka}$ is a semi-orthogonal decomposition. Similarly, $\kt=\sod{\ka, \lorth\ka}$ is a semi-orthogonal decomposition if $\ka$ is left admissible, where $\lorth\ka$ is defined in the obvious way.

Examples typically arise from so-called exceptional objects. Recall that an object $E\in \Db(Z)$ (or any $\IC$-linear triangulated category) is called \emph{exceptional} if $\Hom(E,E)=\IC$ and $\Hom(E,E[k])=0$ for all $k\neq 0$. The smallest triangulated subcategory containing $E$ is then equivalent to $\Db(\Spec(\IC))$ and this category, by abuse of notation again denoted by $E$, is admissible, leading to a semi-orthogonal decomposition $\Db(Z)=\sod{ E\orth, E}$. A sequence of objects $E_1,\ldots,E_n$ is called an \emph{exceptional collection} if
  $\Db(Z) = \sod{ (E_1,\ldots,E_n)\orth,E_1,\ldots,E_n }$
and all $E_i$ are exceptional.
The collection is called \emph{full} if $(E_1\ldots,E_n)\orth=0$.

Note that any fully faithful FM transform $\Phi\colon \Db(X)\to \Db(X')$ gives a semi-orthogonal decomposition $\Db(X')=\sod{ \Phi(\Db(X))\orth,\Phi(\Db(X))}$, because any FM transform has a right and a left adjoint, see \cite[Prop.~5.9]{Huy}.

\subsection{Dual semi-orthogonal decompositions}
Let $\cT$ be a triangulated category together with a semi-orthogonal decomposition $\cT=\sod{\cA_1,\dots,\cA_n}$ such that all $\cA_i$ are admissible. Then there is the \textit{left-dual} semi-orthogonal decomposition $\cT = \sod{ \cB_n,\dots,\cB_1}$ given by $\cB_i \coloneqq \sod{\cA_1,\dots,\cA_{i-1}, \cA_{i+1},\dots,\cA_n}\orth$. There is also a right-dual decomposition but we will always use the left-dual and refer to it simply as the \textit{dual} semi-orthogonal decomposition. We summarise the properties of the dual semi-orthogonal decomposition needed later on in the following 

\begin{lemma}\label{lem:dual-sod}
Let $\cT=\sod{\cA_1,\dots,\cA_n}$ be a semi-orthogonal decomposition with dual semi-orthogonal decomposition $\cT=\sod{\cB_n,\dots,\cB_1}$. 
\begin{enumerate}
\item $\sod{\cA_1,\dots,\cA_r}=\sod{\cB_r,\dots,\cB_1}$ and $\sod{\cA_1,\dots,\cA_r}\orth = \sod{\cB_n,\dots,\cB_{r+1}}$ for $1\le r\le n$.
\item If $\sod{\cA_1,\dots,\cA_n}$ is given by an exceptional collection, i.e.\ $\cA_i=\sod{E_i}$, then its dual is also given by an exceptional collection $\cB_i=\sod{F_i}$ such that $\Hom^*(E_i,F_j)=\delta_{ij}\IC[0]$.
\end{enumerate}
\end{lemma}

\begin{proof}
Part (i) is \cite[Prop.~2.7(i)]{Efimov}. Part (ii) is then clear.
\end{proof}

An important classical example is the following

\begin{lemma}\label{lem:easypdual}
There are dual semi-orthogonal decompositions 
\begin{align*}
\Db(\IP^{n-1}) &= \sod{ \reg,\reg(1),\dots,\reg(n-1) } \,, \\
\Db(\IP^{n-1}) &= \sod{ \Omega^{n-1}(n-1)[n-1],\dots,\Omega^1(1)[1],\reg } \,.
\end{align*}
\end{lemma}

\begin{proof}
The fact that both sequences are indeed full goes back to Beilinson, see \cite[Sect.~8.3]{Huy} for an account. The fact that they are dual is classical and follows by a direct computation, for instance using \cite[Lem.~2.5]{Bri-Ste}.
\end{proof}

The following relative version is the example of dual semi-orthogonal decompositions which we will need throughout the text.

\begin{lemma}\label{lem:Pdual}
Let $\nu\colon Z\to S$ be a $\IP^{n-1}$-bundle. There is the 
semi-orthogonal decomposition
\begin{align*}
  \Db(Z) &= \Sod{ \nu^*\Db(S), \nu^*\Db(S) \otimes \reg_\nu(1), \dots, \nu^*\Db(S)\otimes \reg_\nu(n-1) }
\shortintertext{whose dual decomposition is given by}
  \Db(Z) &=
\Sod{ \nu^*\Db(S)\otimes \Omega_\nu^{n-1}(n-1), \dots, \nu^*\Db(S)\otimes \Omega_\nu^{1}(1), \nu^*\Db(S) }
\, .
\end{align*}
\end{lemma}

\begin{proof}
Part (i) is \cite[Thm.~2.6]{Orlov-projbund}. Part (ii) follows from \autoref{lem:easypdual}.
\end{proof}

The following consequence will be used in \autoref{sub:n-geq-m}.

\begin{corollary}\label{cor:sodequal}
If $m<n$, there is the equality of subcategories of $\Db(Z)$
\begin{align*}
       & \Sod{ \nu^*\Db(S)\otimes \reg_\nu(m-n),\dots, \nu^*\Db(S)\otimes \reg_\nu(-1) } \\
  =\;  & \Sod{ \nu^*\Db(S)\otimes \Omega_\nu^{n-1}(n-1), \dots, \nu^*\Db(S)\otimes \Omega_\nu^{m}(m) }  \, . 
\end{align*}
\end{corollary}

\begin{proof}
Applying \autoref{lem:dual-sod}(i) to the dual decompositions of \autoref{lem:Pdual} gives the equalities
\begin{align*}
      & \Sod{\nu^*\Db(S)\otimes \Omega_\nu^{n-1}(n-1), \dots, \nu^*\Db(S)\otimes \Omega_\nu^{m}(m)} \\
  =\; & \Sod{\nu^*\Db(S),\dots,\nu^*\Db(S)\otimes\reg_\nu(m-1)}^\bot \\
  =\; & \Sod{\nu^*\Db(S)\otimes\reg_\nu(m-n),\dots,\nu^*\Db(S)\otimes\reg_\nu(-1)}\,.\qedhere
\end{align*}
\end{proof}

\subsection{Linear functors and linear semi-orthogonal decompositions} \label{sec:relative}
Let $\cT$ be a tensor triangulated category, i.e.\ a triangulated category with a compatible symmetric monoidal structure. Moreover, let $\cX$ be a \emph{triangulated module category over $\cT$}, i.e.\ there is an exact functor $\pi^*\colon \cT\to \cX$ and a tensor product
  $\otimes \colon \cT \times \cX \to \cX$,
that is an assignment $\pi^*(A)\otimes E$ functorial in $A\in\cT$ and $E\in\cX$.

We will take $\cT=\DperfG(Y)$ for some variety $Y$ with an action by a finite group $G$. Note that $\DperfG(Y)$ has a (derived) tensor product, and it is compatible with $G$-linearisations.

For $\cX$, we have several cases in mind.
If $X$ is a smooth $G$-variety $X$ with a $G$-equivariant morphism $\pi\colon X\to Y$, then we take $\cX=\Db_G(X)=\DperfG(X)$; this is a tensor triangulated category itself and $\pi^*$ preserves these structures.

If $\Lambda$ is a finitely generated $\reg_Y$-algebra, then let $\cX=\Db(\Lambda)$ be the bounded derived category of finitely generated right $\Lambda$-modules with $\pi^*(A)=A\otimes_{\reg_Y} \Lambda$ and
  $\pi^*(A)\otimes E = A\otimes_{\reg_Y}\Lambda\otimes_\Lambda E = A\otimes_{\reg_Y}E\in\cX$.
Note that if $\Lambda$ is not commutative, then $\cX$ is not a tensor category.

We say that a full triangulated subcategory $\cA\subset \cX$ is 
\textit{$\cT$-linear} (since in our cases we have $\cT=\Dperf(Y)$ we will also speak of \textit{Y-linearity}) if
\[
   \pi^*(A)\otimes E\in \cA \qquad \text{for  all $A\in \cT$ and $E\in \cA$.} 
\]
We say that a semi-orthogonal decomposition $\cX = \sod{\cA_1,\dots,\cA_n}$ is \textit{$\cT$-linear}, if all the $\cA_i$ are $\cT$-linear subcategories.

We call a class of objects $\cS\subset \cX$ (left/right) \textit{spanning over $\cT$} if $\pi^*\cT\otimes \cS$ is a (left/right) spanning class of $\cX$ in the non-relative sense.
Recall that a subset $\cC \subset \cX$ is \emph{generating} if $\cX=\Sod{\cC}$ is the smallest triangulated category closed under direct summands containing $\cC$. The subset $\cC \subset \cX$ is called \emph{generating over $\cT$} if $\cC \otimes \pi^*\cT$ generates $\Db(\cX)$.

Let $\cX'$ be a further tensor triangulated category together with a tensor triangulated functor $\pi'^*\colon \cT\to \cX'$. We say that an exact functor $\Phi\colon \cX\to \cX'$ is \textit{$\cT$-linear} if there are functorial isomorphisms 
\[
   \Phi(\pi^*(A)\otimes E) \cong \pi'^*(A)\otimes \Phi(E)
   \qquad \text{for all $A\in \cT$ and $E\in \cX$.} 
\]
The verification of the following lemma is straight-forward.

\begin{lemma} \makeatletter\hyper@anchor{\@currentHref}\makeatother \label{lem:relativespanninggeneral}
\begin{enumerate}
\item If $\Phi\colon \cX\to \cX'$ is $\cT$-linear, then $\Phi(\cX)$ is a
      $\cT$-linear subcategory of $\cX'$.
\item Let $\cA\subset \Db(\cY)$ be a $\cT$-linear (left/right) admissible subcategory.
      Then the essential image of $\cA$ is $\Db(\cY)$ if and only if $\cA$ contains a
      (left/right) spanning class over $\cT$. 
\end{enumerate}
\end{lemma}

For the following, we consider the case that $\cX=\Db(X)$ for some smooth variety $X$ together with a proper morphism $\pi\colon X\to Y$.

\begin{lemma}\label{lem:relativesod}
Let $\cA, \cB\subset \Db(\cX)$ be $Y$-linear full subcategories. Then 
\[
   \cA\subset \cB^\perp  \iff  \pi_*\sHom(B,A)=0 \quad \forall\, A\in \cA, B\in \cB
\,. \]
\end{lemma}

\begin{proof}
The direction $\Longleftarrow$ follows immediately from
 $\Hom^*(B,A) \cong \Gamma(Y, \pi_*\sHom(B,A))$; recall that all our functors are the derived versions.

Conversely, assume that there are $A\in \cA$ and $B\in \cB$ such that $\pi_*\sHom(B,A)\neq 0$. Since $\Dperf(Y)$ spans $\D(\QCoh(Y))$, this implies that there is an $E\in \Dperf(Y)$ such that 
\begin{align*}
  0 \neq  \Hom^*(E, \pi_*\sHom(B,A))\cong \Gamma(Y, \pi_*\sHom(B,A)\otimes E^\vee)
  &\cong  \Gamma(Y, \pi_*(\sHom(B,A)\otimes \pi^*E^\vee)) \\
  &\cong  \Gamma(Y, \pi_*\sHom(B\otimes \pi^*E, A)) \\
  &\cong  \Hom^*(B\otimes \pi^*E,A) \,.
\end{align*}
By the $Y$-linearity, we have $B\otimes \pi^*E\in \cB$ and hence $\cA\not\subset \cB^\perp$. 
\end{proof}

\subsection{Relative Fourier--Mukai transforms}
Let $\pi\colon X \to Y$ and $\pi'\colon X'\to Y$ be proper morphisms of varieties with $X$ and $X'$ being smooth. We denote the closed embedding of the fibre product into the product by $i\colon X\times_YX' \into X\times X'$.
 
We call $\Phi\colon \Db(X)\to \Db(X')$ a \emph{relative FM transform} if $\Phi = \FM_{\iota_*\cP}$ for some object $\cP\in\Db(X\times_YX)$. It is a standard computation that a relative FM transform is linear over $Y$, with respect to the pullbacks $\pi^*$ and $\pi'^*$. Furthermore, we have $\Phi\cong p_*(q^*(\_)\otimes \cP)$ where $p$ and $q$ are the projections of the fibre diagram
\begin{equation}
\begin{aligned}\label{diag:fibre} \xymatrix{
 & X\times_Y X' \ar[dr]^{p} \ar[dl]_q & \\
X \ar[dr]_{\pi}&   & X'\,. \ar[dl]^{\pi'} \\
      & Y &
}
\end{aligned}
\end{equation}

The right adjoint of $\Phi$ is given by $\Phi^R \coloneqq q_*(p^!(\_)\otimes \cP^\vee)\colon \Db(X')\to \Db(X)$. We also have the following slightly stronger statement which one could call \textit{relative adjointness}.

\begin{lemma}
For $E\in \Db(X)$ and $F\in \Db(X')$, there are functorial isomorphisms 
\[
   \pi'_*\sHom(\Phi(E), F)\cong \pi_*\sHom(E,\Phi^R(F))
\,. \]
\end{lemma}

\begin{proof}
 This follows by Grothendieck duality, commutativity of \eqref{diag:fibre}, and projection formula:
\begin{align*}
 \pi'_*\sHom(\Phi(E), F)\cong \pi'_*\sHom(p_*(q^*E\otimes \cP), F) &\cong \pi'_*p_*\sHom(q^*E\otimes \cP, p^!F)\\
&\cong \pi_*q_*\sHom(q^*E, p^!F\otimes \cP^\vee)\\
&\cong \pi_*\sHom(E, q_*(p^!F\otimes \cP^\vee))\\
&\cong \pi_*\sHom(E,\Phi^R( F))\,.\qedhere
\end{align*}
\end{proof}

For $E,F\in \Db(X)$, using the isomorphism of the previous lemma, we can construct a natural morphism $\widetilde \Phi\colon \pi_*\sHom(E,F)\to \pi'_*\sHom(\Phi(E),\Phi(F))$ as the composition
\begin{align}\label{eq:wPhi}
  \widetilde \Phi = \widetilde \Phi(E,F) \colon
  \pi_*\sHom(E,F)\to\pi'_*\sHom(E,\Phi^R\Phi(F))\cong \pi'_*\sHom(\Phi(E),\Phi(F))
\end{align}
where the first morphism is induced by the unit of adjunction $F\to \Phi^R\Phi(F)$. Note that taking global sections gives back the functor $\Phi$ on morphisms, i.e.\ $\Phi=\Gamma(Y,\widetilde \Phi)$ as maps
\[
  \Hom^*(E,F) \cong \Gamma(Y,\pi_*\sHom(E,F)) \to
  \Gamma(Y,\pi'_*\sHom(\Phi(E),\Phi(F))) \cong \Hom^*(\Phi(E),\Phi(F)) \,.
\]
More generally, $\Phi$ induces functors 
for open subsets $U\subseteq Y$, 
\[
   \Phi_U \colon \Db(W) \to \Db(W'), \qquad\text{where }
   W = \pi^{-1}(U)\subseteq X \text{ and } W' = \pi'^{-1}(U)\subseteq X' ,
\]
given by restricting the FM kernel $\iota_*\cP$ to $W\times W'$ and we have $\Phi_U=\Gamma(U, \widetilde\Phi)$. From this we see that $\widetilde \Phi$ is compatible with composition which means that the following diagram, for $E,F,G\in \Db(X)$, commutes
\begin{equation}
\begin{aligned}\label{diag:relativecomposition}
\xymatrix{
  \pi_*\sHom(F,G) \otimes \pi_*\sHom(E,F) \ar[r]\ar[d]_{\widetilde \Phi\otimes \widetilde\Phi}
& \pi_*\sHom(E,G) \ar[d]^{\widetilde \Phi}  \\
  \pi'_*\sHom(\Phi(F),\Phi(G)) \otimes \pi'_*\sHom(\Phi(E),\Phi(F)) \ar[r]
& \pi'_*\sHom(\Phi(E),\Phi(G))  \,.
} 
\end{aligned}
\end{equation}

\subsection{Relative tilting bundles} \label{sub:relative-tilting}
Let $\pi\colon X \to Y$ be a proper morphism of varieties and let $X$ be smooth.
Later on, $X$ and $Y$ will have $G$-actions, and $\Db(X)$ will be replaced by $\Db_G(X)$.

We say that $V\in\Coh(X)$ is a \emph{relative tilting sheaf} if $\Lambda_V \coloneqq \Lambda \coloneqq \pi_*\sHom(V,V)$ is cohomologically concentrated in degree 0 and $V$ is a spanning class over $Y$. For a more general theory of relative tilting bundles, see \cite{BB-tilting}. Note that $\Lambda$ is a finitely generated $\reg_Y$-algebra. We denote the bounded derived category of coherent right modules over $\Lambda$ by $\Db(\Lambda)$. It is a triangulated module category over $\Dperf(Y)$ via $\pi^*A = A\otimes_{\reg_Y}\Lambda$, and $\Lambda$ is a relative generator.
In particular, for $A\in\Db(X)$ and $M\in\Db(\Lambda)$, the tensor product $A\otimes M$ is over the base $\reg_Y$. 

The functor $\pi_*\sHom(V,\_)\colon \Coh(X)\to \Coh(Y)$ factorises over $\Coh(\Lambda)$. Since it is left exact, we can consider its right-derived functor $\pi_*\sHom(V,\_) \colon \Db(X) \to \Db(\Lambda)$. This yields a relative tilting equivalence:

\begin{proposition} \label{prop:relative-tilting}
Let $V\in\Db(X)$ be a relative tilting sheaf over $Y$. Then $V$ generates $\Db(X)$ over $Y$, and the following functor defines a $Y$-linear exact equivalence:
\[
   t_V \coloneqq \pi_*\sHom(V,\_) \colon \Db(X) \isomor \Db(\Lambda)  \,.
\]
\end{proposition}

\begin{proof}
The $Y$-linearity of $t_V$ is due to the projection formula
\[
      t_V(\pi^*A\otimes E)
    = \pi_*(\pi^*A\otimes\sHom(V,E))
\cong A\otimes\pi_*\sHom(V,E)
    = A\otimes t_V(E)  \,.  
\]
Consider the restricted functor $t'_V\colon \cV \coloneqq \sod{V\otimes\pi^*\Dperf(Y)} \to \Db(\Lambda)$. We show that $t'_V$ is fully faithful, using the adjunctions $\pi^*\dashv\pi_*$ and $\_\otimes_{\reg_Y}\Lambda \dashv \mathsf{For}$ where $\mathsf{For}\colon \Db(\Lambda) \to \Db(Y)$ is scalar restriction, the projection formula, and the $Y$-linearity of $t'_V$:
\begin{align*}
       \Hom_{\reg_X}(\pi^*A\otimes V,\pi^*B\otimes V) 
&\cong \Hom_{\reg_X}(\pi^*A, \pi^*B\otimes\sHom(V,V)) \\
&\cong \Hom_{\reg_Y}(A, \pi_*(\pi^*B\otimes\sHom(V,V))) \\
&\cong \Hom_{\reg_Y}( A, B\otimes\Lambda) \\
&\cong \Hom_{\Lambda}(A\otimes \Lambda, B\otimes \Lambda) \\
&\cong \Hom_{\Lambda}(t'_V(\pi^*A\otimes V), t'_V(\pi^*B\otimes V))
\end{align*}
Since objects of type $ \pi^*A\otimes V$ generate $\cV$, this shows that $t'_V$ is fully faithful. We have $t'_V(V)=\Lambda$. Since $\Lambda$ is a relative generator, hence a relative spanning class, of $\Db(\Lambda)$,  we get an equivalence $\cV\cong\Db(\Lambda)$; see \autoref{lem:relativespanninggeneral}.

We now claim that the inclusion $\cV\hookrightarrow \Db(X)$ has a right adjoint, namely
\[
   t'^{-1}_Vt_V\colon \Db(X) \to \Db(\Lambda)\to \cV  \,.
\]
For this, take $A\in\Dperf(Y)$, $F\in \Db(X)$ and compute
\begin{align*}
       \Hom_{\reg_X}(\pi^*A\otimes V,F)
 \cong \Hom_{\reg_Y}(A, \pi_*\sHom(V,F))
&\cong \Hom_\Lambda(A\otimes \Lambda, t_V(F)) \\
&\cong \Hom_{\cV}(t'^{-1}_V(A\otimes \Lambda),t'^{-1}_Vt_V(F)) \\
&\cong \Hom_{\cV}(\pi^*A\otimes V,t'^{-1}_Vt_V(F))   
\end{align*}
where we use the projection formula, the adjunction $\Lambda\otimes_{\reg_Y}\_ \dashv \mathsf{For}$, the fact that $t'^{-1}_V$ is an equivalence, hence fully faithful, and the $Y$-linearity of $t'^{-1}_V$. 

Since the right-admissible $Y$-linear subcategory $\cV\subset \Db(X)$ contains the relative spanning class $V$, we get $\cV=\Db(X)$ by \autoref{lem:relativespanninggeneral}. This shows that $V$ is a relative generator and that $t_V=t_V'$ is an equivalence.   
\end{proof}

Let $\pi'\colon X'\to Y$ be a second proper morphism and let $\Phi\colon\Db(X)\isomor\Db(X')$ be a relative FM transform.

\begin{lemma} \label{lem:relative-tiling--commutation}
If 
\[
  \wPhi_\Lambda \coloneqq \wPhi(V,V) \colon
  \Lambda_V = \pi_*\sHom(V,V) \to \pi'_*\sHom(\Phi(V),\Phi(V)) = \Lambda_{\Phi(V)}
\]
is an isomorphism, then the following diagram of functors commutes:
\begin{equation}
\begin{aligned}\label{diag:t} \xymatrix@C=4em{
  \Db(X)  \ar[r]^{t_V} \ar[d]_\Phi & \Db(\Lambda_V) \ar[d]^{\_\otimes_{\Lambda_V}\Lambda_{\Phi(V)}} \\
  \Db(X') \ar[r]_{t_{\Phi(V)}}      & \Db(\Lambda_{\Phi(V)})
}
\end{aligned}
\end{equation}
\end{lemma}

\begin{proof}
We first show that $\wPhi(V, E)\colon t_V(E)\to t_{\Phi(V)}(\Phi(E))$ is an isomorphism in $\Db(Y)$ for every $E\in \D(X)$. Assume first that there is an exact triangle $\pi^*A\otimes V\to E\to \pi^*B\otimes V$ for some $A,B\in \Dperf(Y)$ and consider the induced morphism of triangles
\[
\xymatrix{
  \pi_*\sHom(V,\pi^*A\otimes V) \ar[r] \ar[d]^{\wPhi(V,\pi^*A\otimes V)} & 
  \pi_*\sHom(V,E) \ar[r] \ar[d]^{\wPhi(V,E)} &  
  \pi_*\sHom(V,\pi^*B\otimes V) \ar[d]^{\wPhi(V,\pi^*B\otimes V)}
  \\
  \pi'_*\sHom(\Phi(V), \Phi(\pi^*A\otimes V)) \ar[r] &
  \pi'_*\sHom(\Phi(V),\Phi(E)) \ar[r] &
  \pi'_*\sHom(\Phi(V), \Phi(\pi^*B\otimes V))
\,.} \]
The outer vertical arrows are isomorphisms because they decompose as
\begin{align*}
           \pi_*\sHom(V,V\otimes\pi^*A)
  &\isomor \pi_*\sHom(V,V)\otimes A \xrightarrow[\wPhi_\Lambda]{\sim} 
           \pi'_*\sHom(\Phi(V),\Phi(V)) \otimes A \\
  &\isomor \pi'_*\sHom(\Phi(V),\Phi(V) \otimes \pi'^*A)
   \isomor \pi'_*\sHom(\Phi(V),\Phi(V \otimes \pi^* A))  \,.
\end{align*}
Therefore, the middle vertical arrow is an isomorphism as well. 
Since $V$ is a relative generator, we can show that $\wPhi(V,E)$ is an isomorphism for arbitrary $E\in \Db(X)$ by repeating the above argument. 

Using the commutativity of \eqref{diag:relativecomposition} with $E$ plugged in for $G$ and $V$ plugged in for all other arguments, we see that $\wPhi(V,E)$ induces an $\Lambda_{\Phi(V)}$-linear isomorphism $\pi_*\sHom(V,E)\otimes_{\Lambda_V}\Lambda_{\Phi(V)}\isomor \pi'_*\sHom(\Phi(V),\Phi(E))$.  
\end{proof}

\begin{lemma} \label{lem:relative-tilting-fully-faithful}
The functor $\Phi$ is fully faithful if and only if $\wPhi_\Lambda\colon \Lambda_V\to \Lambda_{\Phi(V)}$ is an isomorphism.
\end{lemma}

\begin{proof}
If $\Phi$ is fully faithful, the unit $\id\to \Phi^R\Phi$ is an isomorphism. Hence, $\wPhi_\Lambda$ is an isomorphism; see \eqref{eq:wPhi}.

Conversely, let $\wPhi_\Lambda$ be an equivalence. By \autoref{lem:relative-tiling--commutation}, we get a commutative diagram
\[ \xymatrix@C=4em{
   \Db(X)        \ar[r]^{t_V} \ar[d]_\Phi  & \Db(\Lambda_V) \ar[d]^{\_\otimes_{\Lambda_V}\Lambda_{\Phi(V)}} \\
   \sod{\Phi(V)} \ar[r]_{t_{\Phi(V)}}      & \Db(\Lambda_{\Phi(V)})
\,. } \]
In this diagram, the horizontal functors are tilting equivalences. The right-hand vertical functor is an equivalence, too, by assumption on $\wPhi_\Lambda$. Hence, $\Phi\colon \Db(X) \to \sod{\Phi(V)}$ is an equivalence, which implies that $\Phi\colon \Db(X) \to \Db(X')$ is fully faithful.
\end{proof}

\begin{lemma} \label{lem:relative-FM}
Let $V\in\Db(X)$ be a relative tilting sheaf, $\Phi\colon\Db(X)\isomor\Db(X)$ a relative FM autoequivalence, and $\nu\colon V\isomor \Phi(V)$ an isomorphism such that
\[
    \wPhi_\Lambda = \nu\circ\_\circ\nu^{-1} \colon
    \pi_*\sHom(V,V) \to \pi_*\sHom(\Phi(V),\Phi(V))  \,,
\]
i.e.\ $\Phi_U(\phi )\circ\nu = \nu\circ\phi$ for all open subsets $U\subset Y$ and $\phi\in\Lambda_V(U)$.
Then there exists an isomorphism of functors $\id\isomor \Phi$ restricting to $\nu$.
\end{lemma}

\begin{proof}
We claim that, under our assumptions, the following diagram of functors commutes 
\begin{equation}
\begin{aligned}\label{diag:id} \xymatrix@C=4em{
\Db(X) \ar[r]^{t_V} \ar[d]_{\id} & \Db(\Lambda_V) \ar[d]^{\_\otimes_{\Lambda_V}\Lambda_{\Phi(V)}} \\
\Db(X) \ar[r]_{t_{\Phi(V)}}       & \Db(\Lambda_{\Phi(V)})
}
\end{aligned}
\end{equation}
We construct a natural isomorphism $\eta\colon t_{\Phi V} \isomor t_V \otimes_{\Lambda_V}\Lambda_{\Phi V}$ as follows. For $E\in\Db(X)$, there is a natural $\reg_Y$-linear isomorphism $\pi_*\sHom(\Phi(V),E) \isomor \pi_*\sHom(V,E) \otimes_{\Lambda_V}\Lambda_{\Phi (V)}$ given by $f\mapsto f\nu \otimes 1$; the inverse map is $g\otimes1\mapsto g\nu^{-1}$. This map is linear over $\Lambda_{\Phi V}$ because, for a local section $\lambda\in \pi_*\sHom(\Phi(V), \Phi(V))$, we have by our assumption, setting $\phi=\wPhi^{-1}(\lambda)$:
\[
    \eta(f\lambda) = f\lambda\nu \otimes 1 = f\nu \wPhi^{-1}(\lambda)\otimes1
  = f\nu\otimes\lambda = (f\nu\otimes1)\lambda
\,. \]
Comparing the diagrams \eqref{diag:id} and \eqref{diag:t} shows that $\Phi\cong\id$.
\end{proof}

\begin{corollary} \label{cor:relative-tilting-functor-iso}
Let $V\in\Db(X)$ be a relative tilting sheaf, $\Phi_1,\Phi_2\colon\Db(X)\isomor\Db(X')$ relative FM equivalences, and $\nu\colon \Phi_1(V)\isomor \Phi_2(V)$ an isomorphism such that
 $\Phi_{2,U}(\phi )\circ\nu = \nu\circ \Phi_{1,U}(\phi)$ for all $\phi\in\Lambda_V(U)$ and $U\subset Y$ open.
Then there exists a isomorphism of functors $\Phi_1\isomor \Phi_2$ restricting to $\nu$.

Moreover, if $V=L_1\oplus\cdots\oplus L_k$ decomposes as a direct sum, then the above condition is satisfied by specifying isomorphisms $\nu_i\colon \Phi_1(L_i) \isomor \Phi_2(L_i)$ inducing functor isomorphisms $\Phi_{1,U} \isomor \Phi_{2,U}$ on the full finite subcategory  $\{L_{1\mid U},\ldots,L_{k\mid U}\}$ of $\Db(\pi^{-1}(U))$. 
\end{corollary}

\begin{remark}
All the results of this subsection remain valid in an equivariant setting, where a finite group $G$ acts on $X$ and $\pi\colon X\to Y$ is $G$-invariant. Then the correct sheaf of $\reg_Y$-algebras is $\Lambda_V=\pi_*^G\sHom(V,V)$.
\end{remark}

\subsection{Spherical functors} \label{sub:spherical}
An exact functor $\varphi\colon \kc\to \kd$ between triangulated categories is called \emph{spherical} if it admits both adjoints, if the cone endofunctor $F[1] \coloneqq \cone(\id_\kc\to\varphi^R\varphi)$ is an autoequivalence of $\kc$, and if the canonical functor morphism $\varphi^R \to F\varphi^L[1]$ is an isomorphism. A spherical functor is called \emph{split} if the triangle defining $F$ is split. The proper framework for dealing with functorial cones are dg-categories; the triangulated categories in this article are of geometric nature, and we can use Fourier--Mukai transforms.
See \cite{ALdg} for proofs in great generality.

Given a spherical functor $\varphi\colon \kc\to \kd$, the cone of the natural transformation $\TT = \TT_\varphi \coloneqq \cone(\varphi\varphi^R \to \id_\kd)$ is called the \emph{twist around $\varphi$}; it is an autoequivalence of $\kd$. 

The following lemma follows immediately from the definition, since an equivalence has its inverse functor as both left and right adjoint.

\begin{lemma}\label{lem:spherical}
Let $\varphi\colon \kc\to \kd$ be a spherical functor and let $\delta\colon \kd\to\kd'$ be an equivalence. Then $\delta\circ\varphi\colon\cC\to\cD' $ is a spherical functor with associated twist functor
 $\TT_{\delta\varphi} = \delta \TT_\varphi \delta\inv$.
\end{lemma}

\section{The geometric setup}\label{sec:setup}

\noindent
Let $X$ be a smooth quasi-projective variety together with an action of a finite group $G$. Let $S\coloneqq\Fix(G)$ be the locus of fixed points. Then $S\subset X$ is a closed subset, which is automatically smooth since, locally in the analytic topology, the action can be linearised by Cartan's lemma, see \cite[Lem.~2]{Cartan-Quot}. Also note that $X/G$ has rational singularities, like any quotient singularity over $\IC$ \cite{Kovacs}.

\begin{condition} \noindent\label{cond:main}
We make strong assumptions on the group action:
\begin{enumerate}
\item $G \cong \mu_m$ is a cyclic group. Fix a generator $g\in G$.
\item Only the trivial isotropic groups $1$ and $\mu_m$ occur.
\item The generator $g$ acts on the normal bundle $N\coloneqq N_{S/X}$ by multiplication with some fixed primitive $m$-th root of unity $\zeta$. 
\end{enumerate}
\end{condition}

Condition (ii) obviously holds if $m$ is prime.

Condition (iii) can be rephrased: there is a splitting $T_{X\mid S}= T_S\oplus N_{S/X}$ because $T_S$ is the subsheaf of $G$-invariants of $T_{X\mid S}$ and we work over characteristic 0. By (iii), this is even the splitting into the eigenbundles corresponding to the eigenvalues $1$ and $\zeta$. 
We denote by $\chi\colon G\to\IC^*$ the character with $\chi(g)=\zeta^{-1}$. Hence, we can reformulate (iii) by saying that $G$ acts on $N$ via $\chi^{-1}$. 

From these assumptions we deduce the following commutative diagram 

\begin{equation}
\begin{aligned}\label{eq:maindiagram}
%
\xymatrix@C=2.0em@R=5ex{%
       \wX \ar@{->>}^p_{\text{blow-up in $S$}}[rrr] \ar@{->>}_q[dd]
 & & & X   \ar@{->>}^\pi[dd] \\
   & Z=\IP(N) \ar_j[ul] \ar^i[dl] \ar@{->>}^-\nu[r] & S\ar^a[ur] \ar_b[dr] \\
       \wY \save[]+<-2.3em,0.1em>*{\wX/G=}\restore \ar@{->>}_\rho^{\text{blow-up in $S$}}[rrr]
 & & & Y   \save[]+<2.2em,0em>*{=X/G}\restore &
}

\end{aligned}
\end{equation}
where $a$, $b$, $i$, and $j$ are closed embeddings and $\pi$ is the quotient morphism. The $G$-action on $X$ lifts to a $G$-action on $\wX$. Since, by assumption, $G$ acts  
diagonally on $N$, it acts trivially on the exceptional divisor $Z=\IP(N)$. In particular, the fixed point locus of the $G$-action on $\wX$ is a divisor. Hence, the quotient variety $\wY$ is again smooth and the  
quotient morphism $q$ is flat due to the Chevalley--Shephard--Todd theorem. Since the composition $\pi\circ p$ is $G$-invariant, it induces the morphism $\rho\colon \wY\to Y$ which is easily seen to be birational, hence a resolution of singularities. The preimage $\rho^{-1}(S)$ of the singular locus is a divisor in $\wY$. Hence, by the universal property of the blow-up, we get a morphism $\wY\to \Bl_SY$ which is easily seen to be an isomorphism.

\subsection{The resolution as a moduli space of $\boldsymbol{G}$-clusters} \label{sub:clusters}
The result of this section might be of independent interest.
Let $X$ be a smooth quasi-projective variety and $G$ a finite group acting on $X$.
A \textit{$G$-cluster} on $X$ is a closed zero-dimensional $G$-invariant subscheme $W\subset X$ such that the $G$-representation $H^0(W,\reg_W)$ is isomorphic to the regular representation of $G$. There is a fine moduli space $\Hilb^G(X)$ of $G$-clusters, called the \textit{$G$-Hilbert scheme}. It is equipped with the \textit{equivariant Hilbert--Chow morphism} $\tau\colon \Hilb^G(X)\to X/G, W\mapsto\supp(W)$, mapping $G$-clusters to their underlying $G$-orbits.

\begin{proposition} \label{prop:Hilbert-resolution}
Let $G$ be a finite cyclic group acting on $X$ such that all isotropy groups are either 1 or $G$, and such that $G$ acts on the normal bundle $N_{\Fix(G)/X}$ by scalars which means that \autoref{cond:main} is satisfied. Then there is an isomorphism 
\[
   \phi\colon \wY\xrightarrow \cong \Hilb^G(X)
   \quad \text{with} \quad \tau\circ \phi=\rho  \,.
\]
\end{proposition}


\begin{proof}
We use the notation from \eqref{eq:maindiagram}.
  One can identify $\wX$ with the reduced fibre product $(\wY\times_Y X)_\red$ which gives a canonical embedding $\wX\subset \wY\times X$. Under this embedding, the generic fibre of $q$ is a reduced free $G$-orbit of the action on $X$. In particular, it is a $G$-cluster. By the flatness of $q$, every fibre is a $G$-cluster and we get the classifying morphism $\phi\colon \wY \to \Hilb^G(X)$ which is easily seen to satisfy $\tau\circ \phi=\rho$.

Let $s\in S$ and $z\in Z$ with $\nu(z)\in s$. Let $\ell\subset N(s)$ be the line corresponding to $z$. Then, one can check that the tangent space of the $G$-cluster  $q^{-1}(i(z))\subset X$ is exactly $\ell$. Hence, the $G$-clusters in the family $\wX$ are all different so that the classifying morphism $\phi$ is injective. For the bijectivity of $\phi$, it is only left to show that the $G$-orbits supported on a given fixed point $s\in S$ are parametrised by $\IP(N(s))$. Let $\xi\subset X$ be such a $G$-cluster. In particular, $\xi$ is a length $m=|G|$ subscheme concentrated in $s$ and hence can be identified with an ideal $I\subset \reg_{X,s}/\fm_{X,s}^m$ of codimension $m$.  
By Cartan's lemma, the $G$-action on $X$ can be linearised in an analytic neighbourhood of $s$. Hence, there is an $G$-equivariant isomorphism
\[
   \reg_{X,s}/\fm_{X,s}^m \cong \IC[x_1,\dots,x_k,y_1,\dots,y_n]/(x_1,\dots,x_k,y_1,\dots,y_n)^m \eqqcolon R
\]
where $G$ acts trivially on the $x_i$ and by multiplication by $\zeta^{-1}$ on the $y_i$. Furthermore, $n=\rank N_{S/X}$ and $k=\rank T_S=\dim X- n$. By assumption, $\reg(\xi)$ is the regular $\mu_m$-representation. In other words, 
\begin{align} \label{eq:eigendecomp}
  \reg(\xi) \cong R/I \cong \chi^0 \oplus \chi \oplus \dots \oplus \chi^{m-1}
\end{align}
where $\chi$ is the character given by multiplication by $\zeta^{-1}$. In particular, $R/I$ has a one-dimensional subspace of invariants. It follows that every $x_i$ is congruent to a constant polynomial modulo $I$. Hence, we can make an identification $\reg(\xi)\cong R'/J$ where $J$ is a $G$-invariant ideal in $R'=\IC[y_1,\dots,y_k]/\fn^m$ where $\fn=(y_1,\dots,y_n)$. The decomposition of the $G$-representation $R'$ into eigenspaces is exactly the decomposition into the spaces of homogeneous polynomials. Hence, an ideal $J\subset R'$ is $G$-invariant if and only if it is homogeneous. Furthermore, \eqref{eq:eigendecomp} implies that 
\[
   \dim_\IC\bigl( \fn^i/(J\cap\fn^i +\fn^{i+1}) \bigr) = 1 \quad \text{for all $i=0,\dots,m-1$}
\]
which means that $\xi$ is curvilinear.  
In summary $\xi$ can be identified with a homogeneous curvilinear ideal $J$ in $R'$. The choice of such a      
$J$ corresponds to a point in $\IP((\fn/\fn^2)\dual) \cong \IP(N(s))$; see \cite[Rem.~2.1.7]{Goettschebook}.

Hence, $\phi$ is a bijection and we only need to show that $\Hilb^G(X)$ is smooth. The smoothness in points representing free orbits is clear since the $G$-Hilbert--Chow morphism is an isomorphism on the locus of these points. 
So it is sufficient to show that
\[
   \Hom^1_{\Db_G(X)}(\reg_\xi,\reg_\xi) = \dim X = n+k
\]
for a $G$-cluster $\xi$ supported on a fixed point. 
Following the above arguments, we have 
\[
   \Hom^*_{\Db_G(X)}(\reg_\xi,\reg_\xi) \cong \Hom^*_{\Db_G(\IA^k\times \IA^n)}(\reg_{\xi'},\reg_{\xi'})
\]
where $G$ acts trivially on $\IA^k$ and by multiplication by $\zeta$ on $\IA^n$. Furthermore, by a transformation of coordinates, we may assume that
\[
   \xi' = V(x_1,\dots,x_k, y_1^m,y_2,\dots,y_{n}) \subset \IA^k\times \IA^n  \,.
\]
We have $\reg_\xi'\cong \reg_0\boxtimes \reg_\eta$ where
\[
   \eta=V(y_1^m,y_2,\dots,y_n)\subset \IA^n  \,.
\]
By K\"unneth formula, we get
\begin{align*}
        \Hom^*_{\Db_G(\IA^k\times \IA^n)}(\reg_{\xi'},\reg_{\xi'})
 &\cong \Hom^*_{\Db(\IA^k)}(\reg_0,\reg_0) \otimes \Hom^*_{\Db_G(\IA^n)}(\reg_\eta,\reg_\eta) \\
 &\cong \wedge^*(\IC^k)\otimes \Hom^*_{\Db_G(\IA^n)}(\reg_\eta,\reg_\eta)  \,.
\end{align*}
Furthermore, $\Hom^0_{\Db_G(\IA^n)}(\reg_\eta,\reg_\eta)\cong H^0(\reg_\eta)^G\cong \IC$. Hence, it is sufficient to show that $\Hom^1_{\Db_G(\IA^n)}(\reg_\eta,\reg_\eta)\cong \IC^n$. 
Note that $\eta$ is contained in the line $\ell=V(y_2,\dots,y_n)$.
On $\ell$ we have the Koszul resolution
\[
   0 \to \reg_\ell \xrightarrow{\cdot y_1^m} \reg_\ell \to \reg_\eta \to 0  \,.
\]
Using this, we compute 
\[
   \Hom^*_{\Db(\ell)}(\reg_\eta,\reg_\eta) \cong \reg_\eta[0] \oplus \reg_\eta[-1]  \,.
\]
Note that the normal bundle of $\ell$, as an equivariant bundle, is given by $N_{\ell/\IA^n}\cong (\reg_\ell\otimes \chi^{-1})^{\oplus n-1}$. By \cite[Thm.~1.4]{AC}, we have
\begin{align*}
         \Hom^*_{\Db(\IA^n)}(\reg_\eta, \reg_\eta)
  &\cong \Hom^*_{\Db(\ell)}(\reg_\eta, \reg_\eta \otimes \wedge^* N_{\ell/\IA^n}) \\
  &\cong \Hom^*_{\Db(\ell)}(\reg_\eta, \reg_\eta)\otimes \wedge^*((\reg_\ell\otimes \chi^{-1})^{\oplus n-1})  \,. 
\end{align*}
Evaluating in degree 1 gives 
\[
 \Hom^1_{\Db(\IA^n)}(\reg_\eta,\reg_\eta)\cong \reg_\eta\oplus (\reg_\eta\otimes \chi^{-1})^{\oplus n-1}\,.
\]
Since, as a $G$-representation, $\reg_\eta\cong \chi^0\oplus\chi^1\oplus\dots\oplus \chi^{m-1}$, we get an $n$-dimensional space of invariants
\[
        \Hom^1_{\Db_G(\IA^n)}(\reg_\eta,\reg_\eta)
  \cong \Hom^1_{\Db(\IA^n)}(\reg_\eta,\reg_\eta)^G
  \cong \IC^{n}
\,.\qedhere \]
\end{proof}

The following lemma is needed later in \autoref{sub:m-geq-n} but its proof fits better into this section.

\begin{lemma}\label{lem:orthogonal-clusters}
Assume that $m=|G|\ge n=\codim(S\hookrightarrow X)$. Let $\xi_1,\xi_2\subset X$ be two different $G$-clusters supported on the same point $s\in S$.
Then $\Hom^*_{\Db_G(X)}(\reg_{\xi_1}, \reg_{\xi_2})=0$.
\end{lemma}

\begin{proof}
By the same arguments as in the proof of the previous proposition we can reduce to the claim that 
\[
   \Hom^*_{\Db_G(\IA^n)}(\reg_{\eta_1},\reg_{\eta_2}) = 0
\]
where $\eta_1=V(y_1^m, y_2,\dots,y_n)$ and $\eta_2=V(y_1,y_2^m,y_3,\dots,y_n)$. Set $\ell_1=V(y_2,\dots, y_n)$, $\ell_2=V(y_1,y_3,\dots,y_n)$, $E=\gen{\ell_1,\ell_2}=V(y_3,\dots,y_n)$ and consider the diagram of closed embeddings
\begin{align*}\xymatrix{
                               & \ell_2 \ar^{\iota_2}[dr] \ar^{r}[d]  &             \\
  \{0\} \ar^{u}[ur] \ar_{v}[dr] & E      \ar^t[r]                    &  \IA^n \, . \\
                               & \ell_1 \ar^{s}[u] \ar_{\iota_1}[ur] 
} \end{align*} 
where $N_t\cong (\reg_E\otimes \chi^{-1})^{\oplus n-2}$. By \cite[Lem.~3.3]{Kru4} (alternatively, one may consult \cite{Gri} or \cite{ACHderivedint} for more general results on derived intersection theory), we get
\begin{equation} \label{eq:gradedHom}
\begin{split} 
       \Hom^*_{\Db(\IA^n)}(\reg_{\eta_1},\reg_{\eta_2})
&=     \Hom^*_{\Db(\IA^n)}(\iota_{1*}\reg_{\eta_1},\iota_{2*}\reg_{\eta_2})\\ 
&\cong \Hom^*_{\Db(\ell_2)}(\iota_2^*\iota_{1*}\reg_{\eta_1},\reg_{\eta_2})\\ 
&\cong \Hom^*_{\Db(\ell_2)}(u_*v^*\reg_{\eta_1},\reg_{\eta_2})\otimes \wedge^*N_{t\mid \ell_2}\\
&\cong \Hom^*_{\Db(\ell_2)}(u_*v^*\reg_{\eta_1},\reg_{\eta_2})\otimes \wedge^*(\reg_{\ell_2}\otimes \chi^{-1})^{\oplus n-2}
\end{split}
\end{equation}
We consider the Koszul resolution 
$
0\to \reg_{\ell_1}\xrightarrow{y_1^m}\reg_{\ell_1}\to \reg_{\eta_1}\to 0 
$
of $\reg_{\eta_1}$.
Note that this is an equivariant resolution when we consider $\reg_{\ell_1}$ equipped with the canonical linearisation since $y_1^m$ is a $G$-invariant function. Applying $u_*v^*$, we get an equivariant isomorphism
\begin{align}\label{eq:uv}
u_*v^*\reg_{\eta_1}\cong \reg_0\oplus \reg_0[1]\,. 
\end{align}
Similarly, we have the equivariant Koszul resolution
$
 0\to \reg_{\ell}\otimes \chi\xrightarrow{\cdot y} \reg_{\ell} \to \reg_0 \to 0
$
of $\reg_0$, where we set $\ell\coloneqq\ell_2$ and $y\coloneqq y_2$. Applying $\Hom(\_, \reg_{\eta_2})$ to the resolution, we get
\[
   0 \to \IC[y]/y^{m} \otimes \xrightarrow{\cdot y} \IC[y]/y^m \otimes \chi^{-1} \to 0 
\]
and taking cohomology yields
\begin{align} \label{eq:Hom}
       \Hom^*_{\Db(\ell_2)}(\reg_{0},\reg_{\eta_2})
 \cong \IC\gen{y^{m-1}}[0] \oplus \IC\gen{1} \otimes \chi^{-1}[-1]
 \cong \reg_0\otimes \chi^{-1}[0] \oplus \reg_0\otimes \chi^{-1}[-1]  \,. 
\end{align}

Plugging \eqref{eq:uv} and \eqref{eq:Hom} into \eqref{eq:gradedHom} gives
\[
      \Hom^*_{\Db(\IA^n)}(\reg_{\eta_1},\reg_{\eta_2})
\cong \bigl(\reg_0\otimes \chi^{-1}[0] \oplus
      \reg_0^{\oplus 2} \otimes \chi^{-1}[-1] \oplus
      \reg_0 \otimes \chi^{-1}[-2]\bigr) \otimes \wedge^*(\chi^{-1})^{\oplus n-2}
\,. \]
The irreducible representations occuring are $\chi^{-1},\chi^{-2},\dots,\chi^{-(n-1)}$, hence the invariants vanish (recall that $m\ge n$).
\end{proof}

\section{Proof of the main result} \label{sec:main-proof}

\noindent
In this section, we will study the derived categories $\Db(\wY)$ and $\Db_G(X)$ in the setup described in the previous section.
In particular, we will prove Theorems \ref{thm:main} and \ref{thm:tensor}. 

We set $n=\codim(S\hookrightarrow X)$ and $m=|G|$, in other words $G=\mu_m$. We consider, for $\alpha\in \IZ/m\IZ$ and $\beta\in\IZ$, the exact functors
\begin{align*}
\Phi\coloneqq p_*\circ q^*\circ\triv &\colon \Db(\wY)\to \Db_G(X)\\
\Psi\coloneqq (-)^G\circ q_*\circ p^*&\colon \Db_G(X)\to \Db(\wY)\\
\Theta_\beta\coloneqq i_*(\nu^*(\_)\otimes \reg_\nu(\beta)) &\colon \Db(S)\to \Db(\wY)\\
\Xi_\alpha \coloneqq (a_*\circ\triv)\otimes \chi^\alpha &\colon \Db(S)\to \Db_G(X).
\end{align*}
With this notation, the precise version of \autoref{thm:main} is

 \begin{theorem} \noindent \label{thm:mainprecise}
\begin{enumerate}
\item
The functor $\Phi$ is fully faithful for $m\ge n$ and an equivalence for $m=n$. For $m>n$, all the functors $\Xi_\alpha$ are fully faithful and there is a semi-orthogonal decomposition 
\[
   \Db_G(X) = \Sod{ \Xi_{n-m}(\Db(S)), \ldots, \Xi_{-1}(\Db(S)), \Phi(\Db(\wY)) }  \,.
\]
\item
The functor $\Psi$ is fully faithful for $n\ge m$ and an equivalence for $n=m$. For $n>m$, all the functors $\Theta_\beta$ are fully faithful and there is a semi-orthogonal decomposition 
\[
   \Db(\wY) = \Sod{ \Theta_{m-n}(\Db(S)), \ldots, \Theta_{-1}(\Db(S)), \Psi(\Db_G(X)) }  \,.
\]
\end{enumerate}
\end{theorem}

\begin{remark}
We will see later in \autoref{lem:canonicalwY} that $K_\wY\le \rho^* K_Y$ for $m\ge n$ and $K_\wY\ge \rho^*K_Y$ for $n\ge m$. Hence, \autoref{thm:mainprecise} is in accordance with the DK-Hypothesis as described in the introduction.  
\end{remark}

For the proof, we first need some more preparations.

\subsection{Generators and linearity}

\begin{lemma}\label{lem:relative-tilting}
The bundle $V \coloneqq \reg_X\otimes\IC[G] = \reg_X \otimes (\chi^0\oplus\cdots\oplus\chi^{m-1})$ is a relative tilting sheaf for $\Db_G(X)$ over $\Dperf(Y)$.
\end{lemma}

\begin{proof}
If $L\in\Pic(Y)\subset\Dperf(Y)$ is an ample line bundle, then so is $\pi^*(L)$.
Hence, $\Db(X)$ has a generator of the form
 $E \coloneqq \pi^*(\reg_Y\oplus L\oplus\cdots\oplus L^{\otimes k})$
for some $k\gg0$; see \cite{Orlov_gen}. 

In particular, $E$ is a spanning class of $\Db(X)$. Using the adjunction $\Res\dashv \Ind\dashv\Res$, it follows that
  $\Ind(E) \cong E\oplus E\otimes \chi \oplus \dots\oplus E\otimes \chi^{m-1}$
is a spanning class of $\Db_G(X)$. Hence, $V=\Ind\reg_X$ is a relative spanning class of $\Db_G(X)$ over $\Dperf(Y)$.

Since $V$ is a vector bundle, so is $\sHom(V,V) = V\dual \otimes V$. The map $\pi$ is finite, hence $\pi_*$ is exact (does not need to be derived). Finally, taking $G$-invariants is exact because we work in characteristic 0. Altogether, $\pi_*^G\sHom(V,V)$ is a sheaf concentrated in degree 0.
\end{proof}

\begin{lemma}\label{Psilinear}
The functors $\Phi$ and $\Psi$, and for all $\alpha,\beta\in\IZ$ the subcategories
\begin{align*}
  \Xi_\alpha(\Db(S)) &= a_*(\Db(S)) \otimes \chi^\alpha \subset \Db_G(X) 
   \quad\text{and}\\
  \Theta_\beta(\Db(S)) &= i_*\nu^*\Db(S) \otimes \reg_\wY(\beta) \subset \Db(\wY)
\end{align*}
are $Y$-linear for $\pi^*\triv\colon \Dperf(Y)\to \Db_G(X)$ and $\rho^*\colon \Dperf(Y)\to \Db(\wY)$, respectively. 
\end{lemma}

\begin{proof}
We first show that $\Phi$ is $Y$-linear. Recall that in our setup this means
\[
   \Phi(\rho^*(E)\otimes F)\cong \pi^*\triv(E)\otimes \Phi(F)
\]
for any $E\in \Dperf(Y)$ and $F\in \Db(\wY)$. But this holds, since
\begin{align*}
       \pi^*\triv(E) \otimes \Phi(F)
&\cong \pi^*\triv(E) \otimes p_*q^*\triv(F) \\
&\cong p_*(p^*\pi^*\triv(E) \otimes q^*\triv(F)) \\
&\cong p_*(q^*\rho^*\triv(E) \otimes q^*\triv(F)) \\
&\cong p_*q^*\triv(\rho^*(E) \otimes F)  \,.
\end{align*}
The proof that $\Psi$ is $Y$-linear is similar and is left to the reader.

The $Y$-linearity of the image categories follows from \autoref{lem:relativespanninggeneral}(i).
\end{proof}

\begin{lemma}\label{lem:relative-spanning-wY}
The set of sheaves $\cS \coloneqq \{\reg_\wwY\} \cup \{ i_{s*}\Omega^r(r) \mid s\in S, r=0,\dots, n-1\}$ forms a spanning class of $\Db(\wY)$ over $Y$, where $i_s\colon \IP^{n-1} \cong \rho^{-1}(s)\hookrightarrow \wY$ denotes the fibre embedding.
\end{lemma}

\begin{proof}
We need to show that $\hat \cS\coloneqq \rho^*\Dperf(Y)\otimes \cS$ is a spanning class of $\Db(\wY)$.
Let $\tilde y\in \wY\setminus Z$. Then $y=\rho(\tilde y)$ is a smooth point of $Y$. Hence, $\reg_y\in \Dperf(Y)$ and 
$\reg_{\tilde y}\in \rho^*\Dperf(Y)=\rho^*\Dperf(Y)\otimes \reg_\wwY\subset \hat \cS$. Thus, an object 
$E\in \Db(\wY)$ with $\supp E\cap(\wY\setminus Z)\neq \emptyset$ satisfies
$\Hom^*(E,\hat\cS)\neq 0 \neq \Hom^*(\hat\cS, E)$; see \cite[Lem.~3.29]{Huy}.

Let now $0\neq E\in \Db(\wY)$ with $\supp E\subset Z$. Then there exists $s\in S$ such that $i_s^*E\neq 0\neq i_s^!E$; see again \cite[Lem.~3.29]{Huy}. Since the $\Omega^r(r)$ form a spanning class of $\IP^{n-1}$, we get by adjunction
$\Hom^*(E,\cS)\neq 0 \neq \Hom^*(\cS, E)$.
\end{proof}

\subsection{On the equivariant blow-up}
Recall that the blow-up morphism $q\colon \wX\to X$ is $G$-equivariant. Let $\cL_{\wwX}\in \Pic^G(\wwX)$ (we will sometimes simply write $\cL$ instead of $\cL_\wwX$) be the equivariant line bundle $\reg_\wwX(Z)$ equipped with the unique linearisation whose restriction to $Z$ gives the trivial action on $\reg_Z(Z)\cong\reg_\nu(-1)$. We consider a point $z\in Z$ with $\nu(z)=s$ corresponding to a line $\ell\subset N_{S/X}(s)$. Then the normal space $N_{Z/\wwX}(z)$ can be equivariantly identified with $\ell$. It follows by \autoref{cond:main} that $N_{Z/\wwX}\cong (\cL_{\wwX}\otimes \chi^{-1})_{|Z}$ 
as an equivariant bundle. Hence, in $\smash{\Coh_G(\wX)}$, there is the exact sequence
\begin{align} \label{eq:exact-wX}
 0\to \cL_{\wX}^{-1}\otimes \chi\to\reg_\wX\to \reg_Z\to 0 
\end{align}
where both $\reg_\wwX$ and $\reg_Z$ are equipped with the canonical linearisation, which is the one given by the trivial action over $Z$. 

\begin{lemma}\label{lem:pushforward-of-p}
For $\ell=0,\dots, n-1$ we have $p_*\cL_{\wwX}^\ell=\reg_X\otimes \chi^\ell$. 
\end{lemma}

\begin{proof}
We have $p_*\reg_{\wwX}\cong \reg_X$, both, $\reg_\wwX$ and $\reg_X$, equipped with the canonical linearisations. 
Hence, the assertion is true for $\ell=0$. By induction, we may assume that $p_*\cL_{\wwX}^{\ell-1}\cong \reg_X\otimes \chi^{\ell-1}$. We tensor \eqref{eq:exact-wX} by $\cL_{\wwX}^\ell$ to get
\[
   0 \to \cL_{\wX}^{\ell-1} \otimes \chi \to \cL_{\wX}^\ell \to \reg_\nu(-\ell) \to 0  \,. 
\]
Since $0\le \ell\le n-1$, we have $p_*\reg_\nu(-\ell)=0$. Hence, we get an isomorphism
\[
       p_* \big( \cL_{\wX}^\ell \big) 
 \cong p_* \big( \cL_{\wX}^{\ell-1}\otimes \chi \big)
 \cong p_* \big( \cL_{\wX}^{\ell-1} \big)\otimes \chi
 \cong \reg_X \otimes \chi^{\ell-1}\otimes \chi
 \cong \reg_X \otimes \chi^\ell
\,.\qedhere
\]
\end{proof}

\begin{lemma}\label{lem:canonical}
The smooth blow-up $p\colon\wX\to X$ has $G$-linearised relative dualising sheaf
\[
   \omega_p \cong \cL_{\wX}^{n-1}\otimes \chi^{1-n} \in \Pic^G(\wX)  \,.
\]
\end{lemma}

\begin{proof}
The non-equivariant relative dualising sheaf of the blow-up is $\omega_p\cong \reg_{\wwX}((n-1)Z)$.
Since $p$ is $G$-equivariant, $\omega_p$ has a unique linearisation such that $p^!=p^*(\_)\otimes \omega_p\colon \Db_G(X)\to \Db_G(\wwX)$ is the right-adjoint of $p_*\colon\Db_G(\wwX)\to \Db_G(X)$. We now compute this linearisation of $\omega_p$.

As the equivariant pull-back $p^*$ is fully faithful, $p^!\colon \Db_G(X)\to \Db_G(\wwX)$ is fully faithful, too. 
Hence, adjunction gives an isomorphism of equivariant sheaves, $p_*\omega_p\cong p_*p^!\reg_X\cong \reg_X$. The claim now follows from \autoref{lem:pushforward-of-p}.
\end{proof}

We denote by $i_s\colon \IP^{n-1} \cong \rho^{-1}(s) \hookrightarrow \wY$ the embedding of the fibre of $\rho$ and by $j_s\colon \IP^{n-1} \cong p^{-1}(s) \hookrightarrow \wX$ the embedding of the fibre of $p$ over $s\in S$.

\begin{lemma}\label{lem:sky}
For $s\in S$ and $r=0,\dots,n-1$, the cohomoloy sheaves of $p^*\reg_s \in \Db_G(\wX)$ are
\[
    \CH^{-r}(p^*\reg_s) \cong  j_{s*}(\Omega^r(r)\otimes \chi^r)  \,.
\] 
\end{lemma}

\begin{proof}
It is well known that, for the underlying non-equivariant sheaves, we have 
$\CH^{-r}(p^*\reg_s)\cong  j_{s*}\Omega^r(r)$; see \cite[Prop.~11.12]{Huy}.
Since the sheaves $\Omega^r(r)$ are simple, i.e.\ $\End(\Omega^r(r))=\IC$, we have 
$\CH^{-r}(p^*\reg_s)\cong  j_{s*}(\Omega^r(r)\otimes \chi^{\alpha_r})$ for some $\alpha_r\in \IZ/m\IZ$. So we only need to show $\alpha_r=r$.

Let $r\in\{0,\ldots,n-1\}$. We have
 $p_!\cL^{-r} \cong p_*(\cL^{-r+n-1}\otimes \chi^{1-n})$
by \autoref{lem:canonical}. Since $-r+n-1 \in \{0,\ldots,n-1\}$, \autoref{lem:pushforward-of-p} gives $p_!\cL^{-r}\cong \reg_X\otimes \chi^{-r}$. By adjunction,
\[
   \IC[0] \cong \Hom^*_{\Db_G(X)}(\reg_X\otimes \chi^{-r},\reg_s\otimes \chi^{-r})
          \cong \Hom^*_{\Db_G(\wwX)}(\cL^{-r}, p^*\reg_s\otimes\chi^{-r})  \,.
\]
By \autoref{lem:Pdual}, for $r\neq v$, we have 
\[
   \Hom^*_{\Db(\wX)}(\reg_\wX(-rZ), j_{s*}\Omega^v(v)) \cong
   \Hom^*_{\Db(\IP^{n-1})}(\reg(r),\Omega^v(v)) = 0   \,.
\]
Using the spectral sequence in $\Db_G(\wX)$
\[
      E_2^{u,v} = \Hom^u(\cL^{-r},\CH^v(p^*\reg_s\otimes\chi^{-r}))
      \:\Rightarrow\:
      E^{u+v} = \Hom^{u+v}(\cL^{-r}, p^*\reg_s\otimes\chi^{-r})
\]
it follows that
\begin{align*}
  \IC[0] &\cong \Hom^*_{\Db_G(\wX)}(\cL^{-r}, p^*\reg_s\otimes\chi^{-r}) \\
         &\cong \Hom^*_{\Db_G(\wX)}(\cL^{-r}, \CH^{-r}(p^*\reg_s)\otimes\chi^{-r})[r] \\
         &\cong \bigl( \Hom^*_{\Db(\IP^{n-1})}(\reg(r),\Omega^r(r))\otimes \chi^{\alpha_r}\otimes\chi^{-r} \bigr)^G[r] \\
         &\cong \bigl( \IC[-r]\otimes\chi^{\alpha_r-r} \bigr)^G[r] 
\end{align*}
where the last isomorphism is again due to \autoref{lem:Pdual}. Comparing the first and last term of the above chain of isomorphisms, we get $\IC\cong (\chi^{\alpha_r-r})^G$ which implies $\alpha_r=r$. 
\end{proof}

\begin{corollary}\label{cor:Psi-sky}
 Let $n\ge m$ and $\ell\in\{0,\ldots,m-1\}$. Let $\lambda\ge 0$ be the largest integer such that $\ell+\lambda m\le n-1$. Then
\[
   \CH^*(\Psi(\reg_s\otimes \chi^{-\ell})) \cong
   i_{s*}\bigl( \bigoplus\limits_{t=0}^{\lambda} \Omega^{\ell+tm}(\ell+tm)[\ell+tm] \bigr)  \,.
\]
\end{corollary}

\begin{proof}
Since the (non-derived) functor $q_*^G\colon \Coh^G(\wX)\to \Coh(\wY)$ is exact, we have
\[
      \CH^{-r}(\Psi(\reg_s \otimes \chi^{-\ell}))0
\cong q_*^G\bigl( \CH^{-r}(p^*\reg_s) \otimes \chi^{-\ell} \bigr)
\]
and the claim follows from \autoref{lem:sky}.
\end{proof}

\subsection{On the cyclic cover}
The morphism $q\colon \wX\to \wY=\wX/G$ is a cyclic cover branched over the divisor $Z$. This geometric situation and the derived categories involved are studied in great detail in \cite{KuzPerrycyclic}. However, we will only need the following basic facts, all of which can be found in \cite[Sect.~4.1]{KuzPerrycyclic}. 

\begin{lemma}\noindent \label{lem:cyclic}
\begin{enumerate}[itemsep=0.75ex]
\item The sheaf of invariants $q_*^G(\reg_\wwX\otimes\chi^{-1})$ is a line bundle which we denote $\cL_{\wwY}^{-1}\in \Pic(\wY)$.
\item $\cL_{\wwY}^m \cong \reg_{\wwY}(Z)$.
\item $q^G_*(\reg_\wwX\otimes \chi^{\alpha}) \cong \cL_\wwY^\alpha$ for $\alpha\in \{-m+1,\dots,0\}$.
\item  $q^*\circ\triv\colon \Db(\wY) \into \Db_G(\wX)$ is fully faithful, due to $q_*^G(\reg_\wwX) \cong \reg_\wwY$.
\item $q^*(\triv(\cL_\wwY))\cong \cL_\wwX$ are isomorphic $G$-equivariant line bundles.
\item In particular, $\cL_{\wwY\mid Z}\cong\cL_{\wwX\mid Z}\cong \reg_\nu(-1)$.
\end{enumerate}
\end{lemma}

\begin{corollary}\label{cor:Psialpha}
$\Psi(\reg_X\otimes \chi^{\alpha}) \cong \cL_{\wwY}^\alpha$ for $\alpha\in \{-m+1,\ldots,0\}$.
\end{corollary}

\begin{lemma}\label{lem:qcanonical}
The relative dualising sheaf of $q\colon \wX\to\wY=\wX/G$ is $\omega_q\cong \reg_\wX((m-1)Z)$.
\end{lemma}

\begin{proof}
Since the $G$-action on $W \coloneqq \wX\setminus Z$ is free, we have $\omega_{q\mid W}\cong \reg_W$. Hence, $\omega_q\cong \reg_{\wX}(\alpha Z)$ for some $\alpha\in \IZ$. We have 
  $\sHom(\reg_Z,\reg_\wX)\cong \reg_Z(Z)[-1]\cong j_*\reg_\nu(-1)[-1]$,
and hence
\[\begin{array}{rcll}
         i_*\reg_{\nu}(-1)[-1]
 &\cong& q_* j_*\reg_\nu(-1)[-1] \\
 &\cong& q_* \sHom(\reg_Z, \reg_\wX) \\
 &\cong& q_* \sHom(\reg_Z, q^*\reg_\wY) \\
 &\cong& q_* \sHom(\reg_Z,q^!\cL_\wY^{-\alpha}) & \text{by \autoref{lem:cyclic}(v)} \\
 &\cong& \sHom(q_*\reg_Z,\cL_\wY^{-\alpha})     & \text{by Grothendieck duality} \\
 &\cong& \reg_Z(Z)\otimes \cL_\wY^{-\alpha}[-1] \\
 &\cong& i_*\reg_\nu(-m+\alpha)[-1]           & \text{by \autoref{lem:cyclic}(ii)+(vi)}
\end{array}\]
and thus we conclude $\alpha= m-1$.
\end{proof}

\begin{remark}
As an equivariant bundle, we have $\omega_q\cong \cL_\wX^{m-1}\otimes \chi$, but we will not use this.
\end{remark}

\begin{corollary}\label{lem:canonicalwY}
We have $\omega_{\wY\mid Z} \cong \reg_\nu(m-n)$. 
\end{corollary}

\begin{proof}
We have $\omega_{\wX\mid Z}\cong \reg_\nu(-n+1)$; compare \autoref{lem:canonical}. Furthermore, $\omega_{\wY\mid Z}\cong (q^*\omega_{\wY})_{| Z}$. Hence,
\[
   \reg_\nu(1-m) \overset{\ref{lem:qcanonical}}\cong
   \omega_{q\mid Z}\cong \omega_{\wX\mid Z}\otimes \omega_{\wY\mid Z}^\vee \cong
   \reg_{\nu}(1-n)\otimes \omega_{\wY\mid Z}^\vee
   \,.\qedhere
\]
\end{proof}

%

\subsection{The case $\boldsymbol{m\ge n}$}\label{sub:m-geq-n}
Throughout this subsection, let $m\ge n$.

\begin{proposition} \label{prop:Hff}\noindent
\begin{enumerate}
  \item If $m>n$, then the functor $\Xi_\alpha$ is fully faithful for any $\alpha\in \IZ/m\IZ$.
  \item Let $m-n\ge 2$ and $\alpha\neq \beta\in \IZ/m\IZ$. Then 
\[
   \Xi_{\beta}^R\Xi_\alpha = 0 \iff 
   \alpha-\beta \in \{\overline{n-m+1},\overline{n-m+2},\dots,\overline{-1}\}  \,.
\] 
\end{enumerate}
\end{proposition}

\begin{proof}
Recall that
  $\Xi_\beta = (a_*\circ\triv(\_)) \otimes \chi^\beta \colon \Db(S)\to \Db_G(X)$.
Hence, the right-adjoint of $\Xi_\beta$ is given by $\Xi_\beta^R\cong (a^!(\_)\otimes \chi^{-\beta})^G$. By \cite[Thm.~1.4 \& Sect.~1.20]{AC},
\begin{align*}
        \Xi_\beta^R\Xi_\alpha
 &\cong \bigl(a^!a_*(\_)\otimes \chi^{\alpha-\beta}\bigr)^G
  \cong \bigl((\_)\otimes \wedge^*N\otimes \chi^{\alpha-\beta}\bigr)^G \\
 &\cong (\_)\otimes \bigl(\wedge^*N\otimes \chi^{\alpha-\beta}\bigr)^G 
\end{align*}
where, by \autoref{cond:main}, the $G$-action on $\wedge^\ell N$ is given by $\chi^{-\ell}$. We see 
that $(\wedge^*N)^G\cong \wedge^0N[0]\cong \reg_S[0]$; here we use that $m>n$. This shows that, in the case $\alpha=\beta$, we have $\Xi_\alpha^R\Xi_\alpha\cong \id$ which proves (i).
Furthermore, since the characters occurring in $\wedge^*N$ are $\chi^0$, $\chi^{-1}$,\dots,$\chi^{-n}$, we obtain (ii) from
\begin{align*}
         \Xi_\beta^R\Xi_\alpha \neq 0
  &\iff  \bigl(\wedge^*N\otimes \chi^{\alpha-\beta}\bigr)^G\neq 0
   \iff  \overline{0} \in \{ \overline{\alpha-\beta},\overline{\alpha-\beta-1},\ldots,\overline{\alpha-\beta-n} \},
\text{ i.e.} \\
         \Xi_\beta^R\Xi_\alpha = 0 &
   \iff  \alpha-\beta \in \{\overline{n+1},\dots,\overline{m-1}\}
                        = \{\overline{n-m+1},\overline{n-m+2},\dots,\overline{-1}\}
\,.\qedhere
\end{align*}
\end{proof}
\begin{corollary}\label{cor:sod}
For $m>n$, there is a semi-orthogonal decomposition 
\[
   \Db_G(X) = \Sod{ \Xi_{n-m}(\Db(S)), \Xi_{n-m+1}(\Db(S)), \dots, \Xi_{-1}(\Db(S)), \cA },
\]
where $\cA = \lorth\Sod{ \Xi_{n-m}(\Db(S)), \Xi_{n-m+1}(\Db(S)), \dots, \Xi_{-1}(\Db(S))}
$.
\end{corollary}

\begin{proposition} \label{prop:ff}
The functor $\Phi = p_*q^*\triv \colon \Db(\wY)\to \Db_G(X)$ is fully faithful. 
\end{proposition}

\begin{proof}
By \cite[Prop.~7.1]{Huy}, we only need to show for $x,y\in \wY$ that
\[
   \Hom^i_{\Db_G(X)}(\Phi(\reg_x),\Phi(\reg_y))
 = \begin{cases}
      \IC \quad & \text{if $x=y$ and $i=0$} \\
      0   \quad & \text{if $x\neq y$ or $i\notin[0,\dim X]$}.                                    
   \end{cases}
\]
By \autoref{prop:Hilbert-resolution}, $\Phi(\reg_x)=\reg_\xi$ for some $G$-cluster $\xi$. 
Hence,
\[
   \Hom^0_{\Db_G(X)}(\Phi(\reg_x),\Phi(\reg_x)) \cong H^0(\reg_\xi)^G \cong \IC  \,.
\]
Furthermore, since $\Phi(\reg_x)$ is a sheaf, the complex $\Hom^*(\Phi(\reg_x),\Phi(\reg_x))$ is concentrated in degrees $0,\dots,\dim(X)$. It remains to prove the orthogonality for $x\neq y$. If $\rho(x)\neq\rho(y)$, the corresponding $G$-clusters are supported on different orbits. Hence, their structure sheaves are orthogonal. If $\rho(x)=\rho(y)$ but $x\neq y$, the orthogonality was shown in \autoref{lem:orthogonal-clusters}.
\end{proof}

\begin{lemma}\label{lem:C}
The functor $\Phi$ factors through $\cA$. 
\end{lemma}

\begin{proof}
By \autoref{cor:sod}, this statement is equivalent to $\Phi^R \Xi_\alpha=0$ for $\alpha\in\{n-m,\ldots,-1\}$ where $\Phi^R\colon \Db_G(X)\to \Db(\wY)$ is the right adjoint of $\Phi$. Since the composition $\Phi^R \Xi_\alpha$ is a Fourier--Mukai transform, it is sufficient to test the vanishing on skyscraper sheaves of points; see \cite[Sect.~2.2]{Kuz}. So we have to prove that
\[
   \Phi^R \Xi_\alpha(\reg_s) \cong \Phi^R(\reg_s \otimes \chi^\alpha) = 0
\]
for every $s\in S$ and every $\alpha\in \{n-m,\ldots,-1\}$. We have $\Phi^R\cong q_*^G p^!$; recall that $q_*^G$ stands for $(\_)^G\circ q_*$.
By \autoref{lem:canonical} together with \autoref{lem:sky}, we have
\[
   \CH^{-r}(p^!\reg_s) \cong  i_{s*}(\Omega^r(r+1-n) \otimes \chi^{r+1-n})
\]
where the non-vanishing cohomologies occur for $r\in\{0,\dots,n-1\}$. Thus, the linearisations of the cohomologies of $p^!(\reg_s\otimes \chi^{\alpha})$ are given by the characters $\chi^\gamma$ for $\gamma \in\{\alpha+1-n,\ldots,\alpha\}$. We see that, for $\alpha\in \{n-m,\ldots,-1\}$, the trivial character does not occur in $\CH^*(p^!\reg_s\otimes \chi^\alpha)$. This implies that $q_*p^!(\reg_s\otimes \chi^\alpha)$ has vanishing $G$-invariants. 
\end{proof}

We denote by $\cB \subset \Db_G(X)$ the full subcategory generated by the admissible subcategories $\Xi_\alpha(\Db(S))$ for $\alpha\in \{n-m,\dots, -1\}$ and $\Phi(\Db(\wY))$. By the above, these admissible subcategories actually form a semi-orthogonal decomposition
\[
   \cB = \Sod{ \Xi_{n-m}(\Db(S)), \Xi_{n-m+1}(\Db(S)), \dots, \Xi_{-1}(\Db(S)), \Phi(\Db(\wY)) }  \,.
\]

\begin{proposition} \label{prop:full}
We have the (essential) equalities $\cB = \Db_G(X)$ and  $\Phi(\Db(\wY)) = \cA$. 
\end{proposition}

For the proof, we need the following

\begin{lemma}
We have $p_*\cL_{\wwX}^r\otimes \chi^{-\lambda}\in \cB$ for $r\in \IZ$ and $\lambda\in\{0,\ldots,m-n\}$. 
\end{lemma}

\begin{proof}
By \autoref{lem:cyclic}, $\cL_{\wwX}^r\cong q^*(\triv(\cL_{\wwY}^r))$. Hence,
\[
   p_*(\cL_{\wwX}^r) \cong p_*q^*(\triv(\cL_{\wY}^r)) = \Phi(\cL_\wY^r)
   \in \Phi(\Db(\wY)) \subset \cB
\]
which proves the assertion for $\lambda=0$. We now proceed by induction over $\lambda$. Tensoring \eqref{eq:exact-wX} by $\cL_{\wX}^r\otimes \chi^{-\lambda}$ and applying $p_*$, we get the exact triangle
\begin{align} \label{eq:ptriangle}
      p_*\cL^{r-1}_{\wX} \otimes \chi^{-(\lambda-1)}
  \to p_*\cL^r_{\wX}\otimes \chi^{-\lambda}
  \to p_*j_*\reg_Z(-r)\otimes \chi^{-\lambda}
  \to 
\end{align}
where $\reg_Z(-r)$ carries the trivial $G$-action.

The first term of the triangle is an object of $\cB$ by induction. Furthermore, by diagram \eqref{eq:maindiagram}, we have $p_*j_*\reg_Z(-r)\cong a_*\nu_*\reg_Z(-r)$. Hence, the third term of 
\eqref{eq:ptriangle} is an object of $a_*\Db_G(S)\otimes\chi^{-\lambda}=\Xi_{-\lambda}(\Db(S))\subset \cB$. Thus, also the middle term is an object of $\cB$ which gives the assertion.  
\end{proof}

\begin{proof}[Proof of \autoref{prop:full}]
The second assertion follows from the first one since, if $\cB=\Db_G(X)$ holds, both, $\Phi(\Db(\wY))$ and $\cA$, are given by the left-orthogonal complement of \[\Sod{ \Xi_{n-m}(\Db(S)), \Xi_{n-m+1}(\Db(S)), \dots, \Xi_{-1}(\Db(S))}\] in $\Db_G(X)$. 

The subcategories $\Xi_\alpha(\Db(S))$ and $\Phi(\Db(\wY))$ of $\Db_G(X)$ are $Y$-linear by \autoref{Psilinear}. Hence, for the equality $\cB=\Db_G(X)$ it suffices to show that
\[
   \reg_X\otimes \chi^\ell\in \cB
 = \Sod{ \Db(S)\otimes \chi^{n-m}, \dots, \Db(S)\otimes \chi^{-1}, \Phi(\Db(\wY)) }
\]
for every $\ell\in \IZ/m\IZ$; see \autoref{lem:relative-tilting}. Combining \autoref{lem:cyclic} and \autoref{lem:pushforward-of-p}, we see that 
\[
      \Phi(\cL_{\wY}^\ell) 
\cong p_*(q^*\triv(\cL_{\wY})^\ell)
\cong \reg_X\otimes \chi^\ell 
\quad \text{for $\ell=0,\dots, n-1$.}
\]
In particular, 
$\reg_X\otimes\chi^\ell\in \Phi(\Db(\wY))\subset \cB$ for $\ell=0,\dots, n-1$.
Setting $r=0$ in the previous lemma, we find that also $\reg_X\otimes \chi^{\ell}$ for $\ell=n-m,\dots,-1$ is an object of $\cB$.       
\end{proof}

Combining the results of this subsection gives \autoref{thm:mainprecise}(i).

\subsection{The case $\boldsymbol{n\ge m}$}\label{sub:n-geq-m}
Throughout this subsection, let $n\ge m$.

\begin{proposition}
Let $n>m$. Then the functors $\Theta_\beta\colon \Db(S)\to \Db(\wY)$ are fully faithful for every $\beta\in \IZ$ and there is a semi-orthogonal decomposition
\[
  \Db(\wY) = \Sod{ \cC(m-n), \cC(m-n+1), \ldots, \cC(-1), \cD }
\]
where $\cC(\ell) \coloneqq \Theta_\ell(\Db(S)) = i_*\nu^*\Db(S) \otimes \reg_\wY(\ell)$ and
\begin{align*}
\cD &= \bigl\{ E \in \Db(\wY) \mid i^*E\in \lorth \Sod{ \nu^*\Db(S)\otimes\reg_\wY(m-n),\dots,\nu^*\Db(S)\otimes\reg_\wY(-1) } \bigr\} \\
    &= \bigl\{ E \in \Db(\wY) \mid i^*E\in \Sod{ \nu^*\Db(S), \dots, \nu^*\Db(S) \otimes \reg_\wY(m-1) }\bigr\}  \,.
\end{align*}
\end{proposition}

\begin{proof}
This follows from \cite[Thm.~1]{KuzLefschetz}. However, for convenience, we provide a proof for our special case. By construction, $\Theta_\beta^R\cong \nu_*\MM_{\reg_\nu(-\beta)} i^!$. We start with the standard exact triangle of functors
$
 \id\to i^!i_*\to \MM_{\reg_Z(Z)}[-1]\to
$
(see e.g.\ \cite[Cor.~11.4]{Huy}). By \autoref{lem:cyclic}, $\reg_Z(Z)\cong \reg_\nu(-m)$, and thus the above triangle induces for any $\alpha,\beta \in \IZ$
\[
     \nu_*\MM_{\reg_\nu(\alpha-\beta)}\nu^*
 \to \Theta_\beta^R\Theta_\alpha
 \to \nu_*\MM_{\reg_\nu(\alpha-\beta-m)}\nu^*
 \to  \,.
\]
By projection formula, we can rewrite this as
\[
     (\_)\otimes\nu_*\reg_\nu(\alpha-\beta)
 \to \Theta_\beta^R\Theta_\alpha
 \to (\_)\otimes \nu_*\reg_\nu(\alpha-\beta-m)
 \to  \,.   
\]
Now, $\nu_*\reg_\nu\cong \reg_S$ and $\nu_*\reg_\nu(\gamma)=0$ for $\gamma\in \{-n+1,\dots, -1\}$. Hence, $\Theta_\beta^R\Theta_\beta\cong \id$ and $\Theta_\beta^R\Theta_\alpha=0$ if $\alpha-\beta\in \{m-n+1,\dots,-1\}$. 
Therefore, we get a semi-orthogonal decomposition
\[
   \Db(\wY) = \Sod{ \cC(m-n), \cC(m-n+1), \ldots, \cC(-1), \cD }  \,.
\]
The description of the left-orthogonal $\cD$ follows by the adjunction $i^*\dashv i_*$.   
\end{proof}

\begin{lemma}
The functor $\Psi\colon \Db_G(X)\to \Db(\wY)$ factors through $\cD$.  
\end{lemma}

\begin{proof}
By \autoref{lem:relative-tilting}, the equivariant bundles $\reg_X,\reg_X\otimes\chi,\ldots\reg_X\otimes\chi^{m-1}$ generate $\Db_G(X)$ over $\Dperf(Y)$, and therefore so do the bundles $\reg_X\otimes\chi^{-m+1},\ldots,\reg_X\otimes\chi^{-1},\reg_X$ obtained by twisting with $\chi^{1-m}$. Hence, it is sufficient to prove that $\Psi(\reg_X\otimes \chi^\alpha)\in \cD$ for $\alpha \in \{-m+1,\ldots,0\}$ as $\Psi$ and $\cD$ are $Y$-linear; see \autoref{Psilinear}.

Indeed, by \autoref{lem:cyclic} we have $i^*\cL_\wY^\alpha = \cL^\alpha_{\wY\mid Z}\cong \reg_\nu(-\alpha)$, hence
\[
   \Psi(\reg_X\otimes \chi^\alpha) \smash{\overset{\ref{cor:Psialpha}}\cong}
   q_*^G(\reg_X\otimes \chi^\alpha) \cong
   \cL_\wY^{\alpha} \in \cD
\quad
\text{for $\alpha \in \{-m+1,\ldots,0\}$.}
\qedhere \]
\end{proof}

\begin{proposition}
The functor $\Psi\colon \Db_G(X)\to \Db(\wY)$ is fully faithful.  
\end{proposition}

\begin{proof}
We first observe that $V \coloneqq \reg_X\otimes\IC[G] = \reg_X \otimes (\chi^0\oplus\cdots\oplus\chi^{m-1})$ is a relative tilting bundle for $\Db_G(X)$ over $\Dperf(Y)$; see \autoref{lem:relative-tilting}.

For the fully faithfulness, we follow \autoref{lem:relative-tilting-fully-faithful}. So we need to show that $\Psi$ induces an isomorphism $\Lambda_V=\pi_*^G\sHom(V,V) \isomor \rho_*\sHom(\Psi(V),\Psi(V))$. In turn, it suffices to consider the direct summands of $V$. Thus, let $\alpha,\beta\in\{-m+1,\ldots,0\}$ and compute
\[\renewcommand{\arraystretch}{1.4} \begin{array}{rcl}
        \pi_*^G \sHom^*( \reg_X\otimes\chi^\alpha, \reg_X \otimes \chi^\beta)
&\cong& \pi_*^G \sHom^*( \reg_X\otimes\chi^{\alpha+n-1}\otimes\chi^{1-n}, \reg_X \otimes \chi^\beta) \\
&\overset{\ref{lem:pushforward-of-p}}{\cong}&
        \pi_*^G \sHom^*( p_*\cL_\wX^{\alpha+n-1}\otimes\chi^{1-n}, \reg_X \otimes \chi^\beta) \\
&\overset{\ref{lem:canonical}}{\cong}&
        \pi_*^G \sHom^*( p_*(\cL_\wX^\alpha\otimes\omega_p), \reg_X \otimes \chi^\beta) \\
&\cong& \pi_*^G p_* \sHom^*_X(\cL_\wX^\alpha \otimes \omega_p, p^!\reg_X \otimes \chi^\beta) \\
&\cong& \pi_*^G p_* \sHom^*_\wX(\cL_\wX^\alpha, p^* \reg_X\otimes \chi^\beta) \\
&\overset{\ref{lem:cyclic}}{\cong}&
        \rho_* q_*^G \sHom^*_\wX(q^*q_*^G(\reg_\wX\otimes \chi^\alpha), \reg_\wX\otimes \chi^\beta) \\
&\cong& \rho_*\sHom^*_\wY(q_*^G(\reg_\wX\otimes \chi^\alpha), q_*^G(\reg_\wX\otimes \chi^\beta)) \\
&=&     \rho_*\sHom^*_\wY(\Psi(\reg_\wX\otimes \chi^\alpha), \Psi(\reg_\wX\otimes \chi^\beta))  \,.\hfill\qedhere
\end{array} \]
\end{proof}

We denote by $\cE\subset \Db(\wY)$ the full subcategory generated by the admissible subcategories $\Psi(\Db_G(X))$ and $\Theta_\ell(\Db(S)) = i_*\nu^*\Db(S)\otimes\reg_\wwY(\ell)$ for $\ell\in \{m-n,\dots, -1\}$. By the above, these admissible subcategories actually form a semi-orthogonal decomposition
\[
  \cE = \Sod{ \Theta_{m-n}(\Db(S)), \ldots, \Theta_{-1}(\Db(S)), \Psi(\Db_G(X)) }
        \subseteq \Db(\wY)  \,.
 \]
 
\begin{proposition}\label{prop:fullB}
We have the (essential) equalities $\cE = \Db(\wY)$ and  $\Psi(\Db_G(X)) = \cD$. 
\end{proposition}

\begin{proof}
Analogously to \autoref{prop:full}, it is sufficient to prove the equality $\cE = \Db(\wY)$. 
As $\cE$ is constructed from images of fully faithful FM transforms (which have both adjoints), it is admissible in $\Db(\wY)$. Therefore, it suffices to show that $\cE$ contains a spanning class for $\Db(\wY)$. Moreover, because all functors and categories involved are $Y$-linear, it suffices to prove that the relative spanning class $\cS$ of \autoref{lem:relative-spanning-wY} is contained in $\cE$. 

We already know that $\reg_\wY\cong \Psi(\reg_X)\in\Psi(\Db_G(X))\subset \cE$. By \autoref{cor:sodequal}, we get for $s\in S$ and $r\in\{m,\ldots,n-1\}$
\[
   i_{s*}\Omega^r(r) \in \Sod{ \Theta_{m-n}(\Db(S)), \ldots, \Theta_{-1}(\Db(S)) } \subset \cE \, .
\]
By \autoref{cor:Psi-sky}, we have, for $\ell\in\{0,\ldots,m-1\}$, an exact triangle
\[
   E \to \Psi(\reg_s\otimes \chi^{-\ell}) \to i_{s*}\Omega^\ell(\ell)[\ell] \to
\]
where $E$ is an object in the triangulated category spanned by $i_{s*}\Omega^r(r)$ for $r\in\{m,\ldots,n-1\}$. In particular, 
the first two terms of the exact triangle are objects in $\cE$. Hence also $i_{s*}\Omega^\ell(\ell)\in \cE$ for $\ell\in\{0,\ldots,m-1\}$. 
\end{proof}

Combining the results of this subsection gives \autoref{thm:mainprecise}(ii).

\subsection{The case $\boldsymbol{m=n}$: spherical twists and induced tensor products}\label{m=nsection}
Throughout this subsection, $m=n$, so that both functors $\Phi$ and $\Psi$ are equivalences. We will show that the functors $\Theta_\beta$ and $\Xi_\alpha$, which were fully faithful in the cases $n>m$ and $m>n$, respectively, are now spherical. 
Furthermore, the spherical twists along these functors allow to describe the transfer of the tensor structure from one side of the derived McKay correspondence to the other. 
We set $\Theta\coloneqq \Theta_0$ and $\Xi\coloneqq \Xi_0$.

\begin{proposition}\label{prop:Xispherical}
For every $\alpha\in \IZ/m\IZ$, the functor $\Xi_\alpha\colon \Db(S)\to \Db_G(X)$ is a split spherical functor with cotwist $\MM_{\omega_{S/X}}[-n]$. 
\end{proposition}

\begin{proof}
Since $\Xi_\alpha\cong \MM_{\chi^\alpha} \Xi$, it is sufficient to prove the assertion for $\alpha=0$; see \autoref{lem:spherical}. Following the proof of \autoref{prop:Hff}, we have 
$\Xi^R\Xi\cong (\_)\otimes (\wedge^*N)^G$
where $G$ acts on $\wedge^\ell N$ by $\chi^{-\ell}$. From $\rank N=n=m=\ord \chi$, we get
\[
       (\wedge^*N)^G
 \cong \reg_S[0] \oplus \det N[-n]
 \cong \reg_S[0]\oplus\omega_{S/X}[-n]  \,
\]
Hence, $\Xi^R\Xi\cong \id\oplus\, C$ with $C\coloneqq \MM_{\omega_{S/X}}[-n]$. Moreover, $\Xi^R\cong C\Xi^L$ follows from $a^!\cong C a^*$.
\end{proof}

We introduce autoequivalences $\MM_{\cL}\colon \Db(\wY)\to \Db(\wY)$ and $\MM_{\chi}\colon \Db_G(X)\to \Db_G(X)$ given by the tensor products with the line bundle $\cL_{\wwY}$ and the character $\chi$, respectively.

\begin{theorem}\label{thm:n=m}
There are the following relations between functors:
\begin{enumerate}
  \item\label{m=n1} $\Psi^{-1}\cong \MM_\chi \Phi \MM_{\cL^{n-1}}$;
  \item $\Psi \Xi\cong \Theta$, in particular, the functors $\Theta_\beta$ are spherical too;
  \item $\TT_\Theta\cong \Psi \TT_\Xi \Psi^{-1}$;
  \item\label{m=n4} $\Psi^{-1} \MM_\cL \Psi \cong   \MM_\chi\TT_\Xi$ and $\Psi^{-1}\MM_{\cL^{-1}}\Psi\cong \TT_\Xi^{\;-1}\MM_{\chi^{-1}}$.
\end{enumerate}
\end{theorem}

\begin{proof}
In the verification of (i), we use $\omega_p \cong \cL_{\wX}^{n-1}\otimes \chi$, from \autoref{lem:canonical} and $m=n$:
\begin{align*}
  \Psi^{-1}\cong \Psi^L                    \cong
  p_!q^*                                   \overset{\ref{lem:canonical}}{\cong}
  p_*\MM_{\cL_{\wX}^{n-1}\otimes \chi}q^*  \overset{\ref{lem:cyclic}}{\cong}
  \MM_\chi p_*q^* \MM_{\cL^{n-1}}          \cong
  \MM_\chi \Phi \MM_{\cL^{n-1}}  \,. 
\end{align*}
For (ii), first note that, since the $G$-action on $Z\subset \wX$ is trivial, we have 
\[
   \Theta\cong i_*\nu^*\cong q_*j_*\nu^*\cong q_*^Gj_*\nu^*\triv  \,. 
\]
Hence, the base change morphism $\theta \colon p^*a_*\to j_*\nu^*$ induces a morphism of functors
\[
       \hat\theta \colon \Psi \Xi
 \cong q_*^Gp^*a_*\triv\to q_*^Gj_*\nu^*\triv
 \cong \Theta 
\]
which in turn is induced by a morphism between the Fourier--Mukai kernels; see \cite[Sect.~2.4]{Kuz}. Hence, it is sufficient to show that $\theta$ induces an isomorphism $\Psi \Xi(\reg_s)\cong \Theta(\reg_s)$ for every $s\in S$; see \cite[Sect.~2.2]{Kuz}. The morphism $\theta$ induces an isomorphism on degree zero cohomology $L^0p^*a_*(\reg_s)\cong \reg_{p^{-1}(a(s))}\cong j_*L^0\nu^*(\reg_s)$. But there are no cohomologies in non-zero degrees for $j_*\nu^*$ since $\nu$ is flat and $j$ a closed embedding. Furthermore, the non-zero cohomologies of $p^*a_*$ vanish after taking invariants; see \autoref{cor:Psi-sky}. Hence, $\hat\theta(\reg_s)$ is indeed an isomorphism.   

The second assertion of (ii) and (iii) are direct consequences of \autoref{prop:Xispherical} and the formula $\Psi\Xi\cong \Theta$; see \autoref{lem:spherical}.

For (iv), it is sufficient to prove the second relation, and we employ \autoref{cor:relative-tilting-functor-iso} with $L_i=\reg_X\otimes\chi^i$; see also \autoref{lem:relative-tilting}. Recall that $\TT_\Xi^{\;-1}=\cone (\id\to \Xi\Xi^L)[-1]$, and $\Xi^L\cong (\_)^Ga^*$. 
For $1\neq \alpha\in \IZ/n\IZ$, we get 
\[
       \Xi^L\MM_{\chi^-1}(\reg_X\otimes \chi^\alpha)
 \cong (\reg_S\otimes \chi^{\alpha-1})^G
 \cong 0
\,. \]
Hence, $\TT_\Xi^{\;-1}\MM_{\chi^{-1}}(\reg_X\otimes \chi^{\alpha})\cong  \reg_X\otimes \chi^{\alpha-1}$.
We have $\Xi^L(\reg_X)=\reg_S$. Therefore, $\Xi\Xi^L(\reg_X)\cong a_*\reg_S$ and $\TT_\Xi^{\;-1}(\reg_X)\cong \cI_S$. In summary,
\[
      \TT_\Xi^{\;-1}\MM_{\chi^{-1}}(\reg_X\otimes \chi^{\alpha})
\cong \begin{cases}
         \reg_X\otimes \chi^{\alpha-1} \quad & \text{for $\alpha\neq 1$,} \\
         \cI_S                         \quad & \text{for $\alpha= 1$.}
      \end{cases}
\]
On the other hand, for $\alpha\in\{-n+1,\ldots,0\}$, we have $\Psi(\reg_X\otimes \chi^\alpha)\cong \cL^{\alpha}$; see \autoref{cor:Psialpha}. 
Hence, we have 
\[
      \Psi^{-1}\MM_{\cL^{-1}}\Psi(\reg_X\otimes \chi^\alpha)
\cong \reg_X \otimes \chi^{\alpha-1} \quad \text{for $\alpha\in\{-n+2,\ldots,0\}$.}
\]
For $\alpha=-n+1$, we use (i) to get
\begin{align*}
        \Psi^{-1}\MM_{\cL^{-1}}\Psi(\reg_X\otimes \chi^{1-n})
  \cong \Psi^{-1}(\cL_{\wY}^{-n})
  \cong \MM_\chi \Phi(\cL_{\wY}^{-1})
  \overset{\ref{lem:cyclic}}\cong
        p_*(\cL_{\wX}^{-1}\otimes \chi)
  \cong \cI_S
\end{align*}
where we get the last isomorphism by applying $p_*$ to the exact sequence \eqref{eq:exact-wX}.

Therefore, for every $\alpha\in \IZ/n\IZ$ we obtain isomorphisms
\[
   \kappa_\alpha \colon F_1(L_\alpha) \coloneqq
            \TT_\Xi^{\;-1} \MM_{\chi^{-1}}(\reg_X\otimes \chi^\alpha)
    \isomor F_2(L_\alpha) \coloneqq \Psi^{-1} \MM_{\cL^{-1}} \Psi(\reg_X\otimes \chi^\alpha)
\,. \]
Finally, we have to check that the isomorphisms $\kappa_\alpha$ can be chosen in such a way that they form an isomorphism of functors $\kappa\colon F_{1,V\mid \{L_0,\dots, L_{n-1}\}}\isomor F_{2,V\mid \{L_0,\dots, L_{n-1}\}}$ over every open set $V\subset Y$. Let $U \coloneqq Y\setminus S \subset Y$ the open complement of the singular locus. We claim that $F_{1,U} \cong \MM_\chi^{-1} \cong F_{2,U}$. This is clear for $F_2 = \TT_\Xi^{\;-1} \MM_\chi^{-1}$. Furthermore, the map $p\colon \wX \to X$ is an isomorphism and $q\colon \wwX \to \wwY$ is a free quotient when restricted to $W \coloneqq \pi^{-1}(U)$. Since also $\cL_\wwX = q_*^G(\reg_\wwX\otimes\chi)$, we get $\Psi_U \cong {\MM^{-1}_\chi}_{|U}$.

Hence, over $W$, the $\kappa_{i\mid W}$ can be chosen functorially. By the above computations, each $\kappa_{i|W}$ is given by a section of the trivial line bundle. As $S$ has codimension at least 2 in $X$, the sections $\kappa_{i|W}$ over $W$ uniquely extend to sections $\kappa_i$ over $X$. The commutativity of the diagrams relevant for the functoriality now follows from the commutativity of the diagrams restricted to the dense subset $W$. 
\end{proof}

The relations of \autoref{thm:n=m} allow to transfer structures between $\Db(\wY)$ and $\Db_G(X)$. For example, we can deduce the formula 
$\Psi\MM_{\chi^{-1}}\Psi^{-1} \cong  \TT_\Theta \MM_{\cL^{-1}}$. Since $\reg_X\otimes \chi^{\alpha}$ for 
$\alpha\in \{ -(n-1),\ldots,0 \}$ form a relative generator of $\Db_G(X)$, their images $\cL^{\alpha}$ under $\Psi$ do as well. Hence, at least theoretically, our formulas give a complete description of the tensor products induced by $\Psi$ (and also $\Phi$) on both sides. 

Note that $\Phi$ and $\Psi$ are both equivalences, but not inverse to each other. 
Hence, they induce non-trivial autoequivalences $\Psi\Phi\in \Aut(\Db_G(X))$ and $\Phi\Psi\in \Aut(\Db(\wY))$.
Considering the setup of the McKay correspondence as a flop of orbifolds as in diagram \eqref{orbiflop}, it makes sense to call them \textit{flop-flop autoequivalences}. These kinds of autoequivalences were widely studied 
for flops of varieties; see \cite{Todaspherical}, \cite{BodBonflops}, \cite{DonovanWemyssNCflops}, \cite{DonovanWemysstwists}, \cite{ADMflop}. The general picture seems to be that the flop-flop autoequivalences can be expressed via spherical and $\IP$-twists induced by functors naturally associated to the centres of the flops. This picture is called the 'flop-flop=twist' principle; see \cite{ADMflop}. The following can be seen as the first instance of an orbifold 'flop-flop=twist' principle which we expect to hold in greater generality. 

\begin{corollary}\label{cor:flopfloptwistcor}
  $\Psi\Phi \cong \TT_\Theta \MM_{\cL^{-n}} \cong \TT_\Theta \MM_{\reg_\wY(-Z)}$. 
\end{corollary}

\begin{remark}
%
Let us assume $m=n=2$ so that $\chi^{-1}=\chi$. Then, for every $k\in \IN$, we get
\begin{align}\label{Scalaformula}
  \Phi(\cL^{-k}) \cong \cI_S^k\otimes \chi^k 
\end{align}
where $\cI_S^k$ denotes the $k$-th power of the ideal sheaf of the fixed point locus. Indeed,
\[
\begin{array}{rcl}
           \Phi(\cL^{-k}) \overset{\ref{thm:n=m}(i)}
   \cong   \MM_\chi\Psi^{-1}(\cL^{-k-1})
  &\cong&  \MM_\chi(\Psi^{-1} \MM_{\cL^{-1}} \Psi)^k(\cL^{-1}) \\
  &\overset{\ref{cor:Psialpha}}\cong&
           \MM_\chi(\Psi^{-1} \MM_{\cL^{-1}} \Psi)^k(\reg\otimes \chi) \\[1ex]
  &\overset{\ref{thm:n=m}(iv)}\cong&
           (\MM_\chi \TT_\Xi^{\;-1})^k(\reg_X) \\[2ex]
  &\cong&  \cI_S^k\otimes \chi^k
\,.
\end{array}
\]
The last isomorphism follows inductively using the short exact sequences
\[
   0 \to \cI_S^{k+1} \to \cI_S^k \to \cI^k_S/\cI_S^{k+1} \to 0
\]
and the fact that the natural action of $\mu_2$ on $\cI^k_S/\cI_S^{k+1}$ is given by $\chi^k$.
Let now $S$ be a surface and $X=S^2$ with $\mu_2$ acting by permutation of the factors. Then $\wY=S^{[2]}$ is the Hilbert scheme of two points and $\cL_{\wY}$ is the square root of the boundary divisor $Z$ parametrising double points.
For a vector bundle $F$ on $S$ of rank $r$, we have
\[
   \det F^{[2]} \cong \cL_{\wY}^{-r}\otimes \cD_{\det F} 
\]
where $F^{[2]}$ denotes the tautological rank $2r$ bundle induced by $F$ and, for $L\in \Pic S$, we put $\cD_L \coloneqq \rho^*\pi_*(L\boxtimes L)^G\in \Pic S^{[2]}$. Hence, by the $\reg_Y$-linearity of $\Phi$, formula \eqref{Scalaformula} recovers the $n=2$ case of \cite[Thm.~1.8]{Scaladiagonal}.   
\end{remark}

\section{Categorical resolutions} \label{sec:categorical-resolutions}

\subsection{General definitions}\label{sub:catresdef}
Recall from \cite{KuzLefschetz} that a \emph{categorical resolution} of a triangulated category $\cT$ is a smooth triangulated category $\widetilde{\cT}$ together with a pair of functors $P_*\colon \widetilde{\cT}\to \cT$ and $P^*\colon \cT^{\perf}\to \widetilde{\cT}$ such that $P^*$ is left adjoint to $P_*$ on $\cT^{\perf}$ and the natural morphism of functors $\id_{\cT^{\perf}}\to P_*P^*$ is an isomorphism. Here, $\cT^{\perf}$ is the triangulated category of perfect objects in $\cT$. 
Moreover, a categorical resolution $(\widetilde{\cT},P_*,P^*)$ is \emph{weakly crepant} if the functor $P^*$ is also right adjoint to $P_*$ on $\cT^{\perf}$.

For the notion of smoothness of a triangulated category see e.g.\ \cite{KuzLuntscatres}. For us it is sufficient to notice that every  admissible subcategory of $\Db(Z)$ for some smooth variety $Z$ is smooth.
In fact, we will always consider categorical resolutions of $\Db(Y)$, for some variety $Y$ with rational Gorenstein singularities, \textit{inside} $\Db(\wY)$ for some fixed (geometric) resolution of singularities $\rho\colon \wY\to Y$. By this we mean an admissible subcategory $\wcT\subset \Db(\wY)$ such that $\rho^*\colon \Dperf(Y)\to \Db(\wY)$ factorises through $\wcT$.

%
%
%

By Grothendieck duality, we get a canonical isomorphism $\reg_Y\cong \rho_*\reg_\wwY\cong \rho_*\omega_\rho$.
This induces a global section $s$ of $\omega_\rho$, unique up to a global unit (i.e.\ scalar multiplication by an element of $\reg_Y(Y)^\times$), and hence a morphism of functors 
\[
   t \coloneqq \rho_*(\_\otimes s)\colon \rho_* \to \rho_!  \,.
\]
Since this morphism can be found between the corresponding Fourier--Mukai kernels, we may define the cone of functors $\rho_+ \coloneqq \cone(t) \colon \Db(\wY) \to \Db(Y)$.

\begin{definition}
The \textit{weakly crepant neighbourhood of $Y$ inside $\Db(\wY)$} is the full triangulated subcategory
\[
   \WC(\rho) \coloneqq \ker(\rho_+) \subset \Db(\wY)  \,.
\]
\end{definition}

\begin{proposition}
If $\WC(\rho)$ is a smooth category (which is the case if it is an admissible subcategory of $\Db(\wY)$), it is a categorical weakly crepant resolution of singularities. 
\end{proposition}

\begin{proof}
By adjunction formula, $t\rho^*\colon \rho_*\rho^*\to \rho_!\rho^*$ is an isomorphism. Hence, $\rho_+\rho^*=0$ and $\rho^*\colon \Dperf(Y)\to \Db(\wY)$ factors through $\WC(\rho)$. By definition, $\rho_!$ is the left adjoint to $\rho^*$. Since $\rho_*$ and $\rho_!$ agree on $\WC(\rho)$, we also have the adjunction $\rho_*\dashv \rho^*$ on $\WC(\rho)$. 
\end{proof}

\begin{remark}
 We think of $\WC(\rho)$ as the biggest weakly crepant categorical resolution inside the derived category $\Db(\wY)$ of a given geometric resolution $\rho\colon \wY\to Y$. The only thing that prevents us from turning this intuition into a statement is the possibility that, for a given weakly crepant resolution $\cT\subset \Db(\wY)$, there might be an isomorphism $\rho_{*\mid \cT}\cong \rho_{!\mid\cT}$ which is not the restriction of $t$ (up to scalars).
\end{remark}

\subsection{The weakly crepant neighbourhood in the cyclic setup}
In the case of the resolution of the cyclic quotient singularities discussed in the earlier sections, $\WC(\rho)$ is indeed a categorical resolution by the following result. We use the notation of \autoref{sec:setup}; recall $G=\mu_m$.

\begin{theorem} \label{thm:cyclicWC}
Let $Y=X/G$, $\rho\colon \wY\to Y$ and $i\colon Z = \rho\inv(S)\hookrightarrow \wY$ be as in \autoref{sec:setup}. Assume $m\mid n=\codim(S\hookrightarrow X)$ and $n>m$. Then
there is a semi-orthogonal decomposition
\[
   \WC(\rho) = \Sod{ i_*(\cE), \Psi(\Db_{\mu_m}(X)) }
\]
where 
\begin{align*}
  \cE = \langle &\cA(-m+1), \cA(-m+2)\dots, \cA(-1), \\
    & \cA\otimes\Omega^{n-m-1}(n-m-1), \cA\otimes\Omega^{n-m-2}(n-m-2), \dots, \cA\otimes\Omega^{m}(m) \rangle
\end{align*}
with $\cA \coloneqq \nu^*\Db(S)$ and $\cA(i) \coloneqq \cA\otimes \reg(i)$; the $\cA\otimes \Omega^i(i)$ parts of the decomposition do not occur for $n=2m$.
In particular, $\WC(\rho)$ is an admissible subcategory of $\Db(\wY)$.                
\end{theorem}

\begin{proof}
We first want to show that $\Psi(\Db_{\mu_m}(X))\subset \WC(\rho)$. For this, by \autoref{lem:relative-tilting}, it is sufficient to show that $\cL_{\wwY}^a=\Psi(\reg_X\otimes \chi^a)\in \WC(\rho)$ for every $a\in\{-m+1,\ldots,0\}$. 

The equivariant derived category $\Db_{\mu_m}(X)$ is a strongly (hence also weakly) crepant categorical resolution of the singularities of $Y$ via the functors $\pi^*, \pi_*^{\mu_m}$; see \cite[Thm.~1.0.2]{Abuaf-catres}. Since $\Psi\circ \pi^*\cong \rho^*$ (see \autoref{Psilinear}), $\cC \coloneqq \Psi(\Db_{\mu_m}(X))$ is a crepant resolution via the functors $\rho^*,\rho_*$. Hence, $\rho_*\cL_{\wwY}^a\cong \rho_!\cL_\wwY^a$ for $a\in \{-m+1,\ldots,0\}$ and it is only left to show that this isomorphism is induced by $t$. Again by the $Y$-linearity of $\Psi$, we have $\rho_*\cL_\wwY^a\cong \pi_*(\reg_X\otimes \chi^a)^{\mu_m}$ which is a reflexive sheaf on the normal variety $Y$ (this follows for example by \cite[Cor.~1.7]{Hartreflexive}). By construction, $t$ induces an isomorphism over $Y\setminus S$. Since the codimension of $S$ is at least 2, $t\colon \rho_*\cL_{\wwY}^a\to \rho_!\cL_{\wwY}^a$ is an isomorphism of reflexive sheaves over all of $Y$; see \cite[Prop.~1.6]{Hartreflexive}.   

By \autoref{thm:mainprecise}(ii), we have $\Db(\wwY)\cong \sod{ \cB,\cC }$ with 
\begin{align*}
 \cB \cong i_*\Sod{ \cA(m-n),\dots,\cA(-1) }
     \cong i_*(\Sod{ \cA,\cA(1),\dots,\cA(m-1) }\orth)  \,.
\end{align*}
We have $\rho_*\cB=0$. It follows that $\WC(\rho) = \sod{ \cB\cap \ker(\rho_!), \cC}$. Indeed, consider an object $A\in \Db(\wwY)$. It fits into an exact triangle
$
  C \to A \to B \to
$
with $C\in \cC$ and $B\in \cB$. From the morphism of triangles
\[ \xymatrix@C=2em{
  \rho_*(C) \ar[r] \ar[d]^{t(C)}_\cong & \rho_*(A) \ar[r] \ar[d]^{t(A)} & \rho_*(B)=0 \ar[r] \ar[d]^{t(B)} & \\
  \rho_!(C) \ar[r]                     & \rho_!(A) \ar[r]               & \rho_!(B) \ar[r] &
} \]
we see that $t(A)$ is an isomorphism if and only if $\rho_!B=0$.

It is left to compute $\cB\cap \ker(\rho_!)$. Let $F\in \Db(Z)$ and $B=i_*F$. By \autoref{lem:canonicalwY},
\[
   \rho_!B \cong \rho_!i_*F \cong b_*\nu_*(F\otimes\reg_\nu(m-n))  \,.
\]       
We see that $B\in \ker \rho_!$ if and only if $\nu_*(F\otimes\reg_\nu(m-n))=0$ if and only if $F\in \nu^*\Db(S)(n-m)\orth$. Hence, $\cB\cap \ker\rho_!= i_*(\cF^\perp)$ with
\[
   \cF = \Sod{\cA, \cA(1), \dots, \cA(m-1), \cA(n-m)} \subset \Db(Z)
\]
Carrying out the appropriate mutations within the semi-orthogonal decomposition
\[
   \Db(Z) = \Sod{\cA(-m+1), \cA(-m+2), \dots, \cA(n-m-1), \cA(n-m)}  \,,
\]
we see that $\cF^\perp = \cE$; compare \autoref{lem:Pdual}.

Since $\cE\subset \Sod{\cA(m-n),\dots, \cA(-1)}$ is an admissible subcategory, we find that $i_*\colon \cE\to \Db(\wY)$ is fully faithful and has adjoints. 
Hence, $\WC(\rho)\subset \Db(\wY)$ is admissible.  
\end{proof}

\begin{remark} \label{rem:SODWCN}
We have $\Db(\wY) = \Sod{ i_*(\cA\otimes \Omega^{n-1}(n-m)), \WC(\rho) }$.
In other words, we can achieve categorical weak crepancy by dropping only one $\Db(S)$ part of the semi-orthogonal decomposition of $\Db(\wY)$.
\end{remark}


\subsection{The discrepant category and some speculation} \label{sub:discrepant-category}
Let $Y$ be a variety with rational Gorenstein singularities and $\rho\colon \wY\to Y$ a resolution of singularities. 
Then, $\rho$ is a crepant resolution if and only if $\Db(\wY) = \WC(\rho)$; compare \cite[Prop.~2.0.10]{Abuaf-catres}.
We define the \textit{discrepant category} of the resolution as the Verdier quotient 
\[
   \disc(\rho) \coloneqq \Db(\wY)/\WC(\rho)  \,.
\]
By \cite[Remark~2.1.10]{Neeman}, since $\WC(\rho)$ is a kernel, and hence a thick subcategory, we have $\disc(\rho)=0$ if and only if $\Db(\wY) = \WC(\rho)$.
Therefore, we can regard $\disc(\rho)$ as a categorical measure of the discrepancy of the resolution $\rho\colon\wY\to Y$.

In our cyclic quotient setup, where $\wY\cong \Hilb^G(X)$ is the simple blow-up resolution, we have
$\disc(\rho)\cong \Db(S)$ by \autoref{rem:SODWCN} and \cite[Lem.~A.8]{LuntsSchnuererSOD}.  
Hence, in this case, $\disc(\rho)$ is the smallest non-zero category that one could expect (this is most obvious in the case that $S$ is a point).
This agrees with the intuition that the blow-up resolution is minimal in some way.

\begin{question}
Given a variety $Y$ with rational Gorenstein singularities, is there a resolution $\rho\colon \wY\to Y$ of minimal categorical discrepancy in the sense that, for every other resolution $\rho'\colon \wY'\to Y$, there is a fully faithful embedding $\disc(\rho)\hookrightarrow \disc(\rho')$?
\end{question}

Often, in the case of a quotient singularity, a good candidate for a resolution of minimal categorical discrepancy should be the $G$-Hilbert scheme.

At least, we can see that $\disc(\rho)$ grows if we further blow up the resolution away from the exceptional locus.

\begin{proposition}\label{prop:smoothblowup}
Let $\rho\colon \wY\to Y$ be a resolution of singularities and let $f\colon \wY'\to \wY$ be the blow-up in a smooth center $C\subset \wY$ which is disjoint from the exceptional locus of $\rho$. Set $\rho' \coloneqq \rho f \colon \wY'\to Y$. Then there is a semi-orthogonal decomposition 
\[
   \disc(\rho') = \Sod{\Db(C), \disc(\rho)}  \,. 
\]
\end{proposition}

We first need the following general

\begin{lemma}\label{lem:Verdiersod}
Let $\cD$ be a triangulated category, $\cC\subset \cD$ a triangulated subcategory, and $\cD=\Sod{\cA,\cB}$ a semi-orthogonal decomposition so that the right-adjoint $i_\cB^!$ of the inclusion $i_\cB\colon \cB\hookrightarrow \cD$ satisfies $i_\cB^!(\cC)\subset \cB\cap \cC$. Then there is a semi-orthogonal decomposition 
\[
    \cD/\cC \cong \Sod{ \cA/(\cA\cap \cC), \cB/(\cB\cap \cC) }  \,.
\]
\end{lemma}

\begin{proof}
For every object $D\in \cD$, we have an exact triangle 
\begin{align} \label{eq:tri}
  i_\cB^!D \to D \to i_\cA^*D \to  
\end{align}
where $i_\cA^*$ is the left-adjoint to the embedding $i_\cA\colon \cA\to \cD$. Considering an object $C\in \cC$ shows that our assumption $i_\cB^!(\cC)\subset \cB\cap \cC$ implies $i_\cA^*(\cC)\subset \cA\cap \cC$.

Let $C\in \cC$ and $A\in \cA$. Then, using the long exact Hom-sequence associated to the triangle \eqref{eq:tri}, we see that every morphism $C\to A$ factors as $C\to \iota_\cA^*C\to A$. Hence, the embedding $i_\cA$ descends to a fully faithful embedding $\bar i_\cA\colon \cA/(\cA\cap \cC)\to \cD/\cC$, by \cite[Prop.~B.2]{LuntsSchnuererSOD} (set $\cW=\cA$, $\cV=\cA\cap \cC$ and use (ff2) of \textit{loc.\ cit.}). Similarly, we get an induced fully faithful embedding $\bar i_\cB\colon \cB/(\cB\cap \cC)\to \cD/\cC$ (use (ff2)\textsuperscript{op} instead of (ff2)).  

Now let us show that $\Hom_{\cD/\cC}\bigl(\cB/(\cB\cap \cC), \cA/(\cA\cap \cC)\bigr)=0$. For $B\in \cB$ and $A\in \cA$, a morphism $B\to A$ in $\cD/\cC$ is represented by a roof
\[
   B \xleftarrow{\beta} D \xrightarrow{\alpha} A 
\]
where $\beta\colon D\to B$ is a morphism in $\cD$ with $\cone(\beta)\in \cC$ and $\alpha\colon D\to A$ is any morphism in $\cD$; see \cite[Def.~2.1.11]{Neeman}. Put $C \coloneqq \cone(\beta)[-1] \in \cC$. We apply the triangle of functors $i^*_\cA \to \id \to i^!_\cB \to$ (formally, $i^*_\cA$ has to be replaced by $i_\cA i^*_\cA$ and $i^!_\cB$ by $i_\cB i^!_\cB$) to the triangle of objects $C\to D\to B \to$ and obtain the diagram
\[ \xymatrix@C=2em@R=3ex{
i^!_\cB C \ar[r] \ar[d] & i^!_\cB D \ar[r] \ar[d]^\gamma & B \ar[d] \\
        C \ar[r] \ar[d] &         D \ar[r] \ar[d] & B \ar[d] \\
i^*_\cA C \ar[r]        & i^*_\cA D \ar[r]        & 0 \\
} \]
where we have used $i^!_\cB B = B$ and $i^*_\cA B = 0$.
Now $i^!_\cB C \in \cC\cap\cB$ by assumption. The left column thus forces
$i^*_\cA C \cong i^*_\cA D \in \cC$. 
We get that $\cone{\beta\gamma}\in \cC$ since $\cone{\beta}, \cone{\gamma}\in \cC$; see \cite[Lem.~1.5.6]{Neeman}.
Therefore, we get another roof representing the same morphism in $\cD/\cC$, replacing $D$ by $i^!_\cB D$:
\[
   B \xleftarrow{\beta\gamma} i^!_\cB D \xrightarrow{\alpha\gamma} A
\,. \]
However, $i^!_\cB D\in\cB$ and $\Hom_\cD(\cB,\cA)=0$, so the morphism is 0 in $\cD/\cC$.

Finally, we need to show that $\cA/(\cA\cap\cC)$ and $\cB/(\cB\cap\cC)$ generate $\cD/\cC$, but this is clear, because $\cA$ and $\cB$ generate $\cD$.
\end{proof}

\begin{proof}[Proof of \autoref{prop:smoothblowup}]
We have a semi-orthogonal decomposition $\Db(\wY')=\Sod{\cA, \cB}$ with 
$\cB = f^*\Db(\wY)$ and 
\[
   \cA = \Sod{ \iota_*(g^*\Db(C)\otimes \reg_g(-c+1)),\dots, \iota_*(g^*\Db(C)\otimes \reg_g(-1))}
\,. \]
Here, $c=\codim(C\into\wY)$ and $g$ and $\iota$ are the $\IP^{n-1}$-bundle projection and the inclusion of the exceptional divisor of the blow-up $f \colon \wY' \to \wY$. 

Let $U \coloneqq \wY\setminus C$.
For $F\in \Db(\wY')$ we have $(f_*F)_{\mid U}\cong (f_!F)_{\mid U}$. Hence, if $F\in \WC(\rho')$, we must have $(f_*F)_{\mid U}\in \WC(\rho_{\mid U})$. Since $\rho$ is an isomorphism in a neighbourhood of $C$, an object $E\in \Db(\wY)$ is contained in $\WC(\rho)$ if and only if its restriction $E_{\mid U}$ is contained in $\WC(\rho_{\mid U})$. In summary, 
\[
   f_* F \in \WC(\rho) \quad\text{for every}\quad F\in \WC(\rho')  \,.
\]
We have $f_*\reg_{\wwY'}\cong \reg_\wwY\cong f_*\omega_f$. By the projection formula, it follows that $f_{*\mid \cB}\cong f_{!\mid \cB}$. Hence, we have $\cB\cap \WC(\rho')=f^*\WC(\rho)$ and $\cB/(\WC(\rho')\cap \cB)\cong\disc(\rho)$.

Now, we can apply \autoref{lem:Verdiersod} with $\cC=\WC(\rho')$ to get a semi-orthogonal decomposition 
\[
   \disc(\rho') = \Sod{ \cA/(\WC(\rho')\cap \cA), \disc(\rho) }  \,.
\]
We have $f_*(\cA)=0$, hence $\rho'_*(\cA)=0$. Accordingly, 
\[
   \WC(\rho') \cap \cA
 = \ker(\rho'_!) \cap \cA
 = \ker(f_!) \cap \cA
\,. \]
The second equality is due to the fact that all objects of $f_!\cA$ are supported on $C$, where $\rho$ is an isomorphism. Now, in analogy to the computations of the proof of \autoref{thm:cyclicWC} and \autoref{rem:SODWCN}, we get a semi-orthogonal decomposition
\[
   \cA = \Sod{ \iota_*(g^*\Db(C)\otimes\Omega^{c-1}(c-1)), \ker(f_!)\cap \cA }  \,.
\]
Hence, $\cA/(\WC(\rho')\cap \cA) \cong \iota_*(g^*\Db(C)\otimes\Omega^{c-1}(c-1)) \cong \Db(C)$.
\end{proof}

\subsection{(Non-)unicity of categorical crepant resolutions} \label{sub:non-unicity}
Let $\wY\to Y$ be a resolution of rational Gorenstein singularities and let $\cD\subset \Dperf(\wY)$ be an admissible subcategory which is a weakly crepant resolution, i.e.\ $\rho^*\colon \Dperf(Y)\to \Db(\wY)$ factors through $\cD$ and $\rho_{*\mid \cD}\cong \rho_{!\mid \cD}$. Then every admissible subcategory $\cD'\subset\cD$ with the property that $\rho^*\colon \Dperf(Y)\to \Db(\wY)$ factors through $\cD'$ is a weakly crepant resolution, too.

In particular, in our setup of cyclic quotients, there is a tower of weakly crepant resolutions of length $n-m$ given by successively dropping the $\Db(S)$ parts of the semi-orthogonal decomposition of $\WC(\rho)$. We see that weakly crepant categorical resolutions are not unique, even if we fix the ambient derived category $\Db(\wY)$ of a geometric resolution $\wY\to Y$.

In contrast, \textit{strongly crepant} categorical resolutions are expected to be unique up to equivalence; see \cite[Conj.~4.10]{KuzLefschetz}. A strongly crepant categorical resolution of $\Db(Y)$ is a module category over $\Db(Y)$ with trivial relative Serre functor; see \cite[Sect.~3]{KuzLefschetz}. For an admissible subcategory $\cD\subset \Db(\wY)$ of the derived category of a geometric resolution of singularities $\rho\colon \wY\to Y$ this condition means that $\cD$ is $Y$-linear and there are  functorial isomorphisms
\begin{align}\label{stronglycrepantcondition}
  \rho_*\sHom(A,B)^\vee \cong \rho_*\sHom(B,A)
\end{align}
for $A,B\in \cD$. In our cyclic setup, $\Psi(\D_G(X))\subset \Db(\wY)$ is a strongly crepant categorical resolution; see \cite[Thm.~1]{KuzLefschetz} or \cite[Thm.~10.2]{Abuaf-catres}. 

We require strongly crepant categorical resolutions to be \textit{indecomposable} which means that they do not decompose into direct sums of triangulated categories or, in other words, they do not admit both-sided orthogonal decompositions.
Under this additional assumption, we can prove that strongly crepant categorical resolutions are unique if we fix the ambient derived category of a geometric resolution.

\begin{proposition} \label{prop:unicity}
 Let $\wY\to Y$ be a resolution of Gorenstein singularities and $\cD, \cD'\subset \Db(\wY)$ admissible indecomposable strongly crepant subcategories. Then $\cD=\cD'$.
\end{proposition}

\begin{proof}
The intersection $\cD\cap \cD'$ is again an admissible $Y$-linear subcategory of $\Db(\wY)$ containing $\rho^*(\Dperf(Y))$. Furthermore, condition \eqref{stronglycrepantcondition} is satisfied for every pair of objects of $\cD\cap \cD'$; so the intersection is again a strongly crepant resolution. Hence, we can assume $\cD'\subset \cD$.

Let $\cA$ be the right-orthogonal complement of $\cD'$ in $\cD$, so that we have a semi-orthogonal decomposition 
$\cD = \Sod{\cA, \cD'}$. By \autoref{lem:relativesod}, this means that $\rho_*\sHom(D,A)=0$ for $A\in \cA$ and $D\in \cD'$. But then, by \eqref{stronglycrepantcondition}, we also get $\rho_*\sHom(A,D)=0$ so that $\cD=\cA\oplus \cD'$. 
\end{proof}

\subsection{Connection to Calabi--Yau neighbourhoods} \label{sub:CY-neighbourhoods}
In \cite{HKP}, \textit{spherelike objects} and their \textit{spherical subcategories} were introduced and studied. The paper hinted at a role of these notions for birationality questions of Calabi--Yau varieties. One of the starting points for our project was to consider \textit{Calabi--Yau neighbourhoods} (a generisation of spherical subcategories) as candidates for categorical crepant resolutions of Calabi--Yau quotient varieties. In this subsection, we describe the connection to the weakly crepant resolutions considered above.  

We recall some abstract homological notions. Let $\cT$ be a Hom-finite $\IC$-linear triangulated category and $E\in\cT$ an object. We say that $\Serre E\in\cT$ is a \emph{Serre dual object} for $E$ if the functors $\Hom^*(E,-)$ and $\Hom^*(-,\Serre E)^\vee$ are isomorphic. By the Yoneda lemma, $\Serre E$ is then uniquely determined. Fix an integer $d$. We call the object $E$
\begin{itemize}
  \item a \emph{$d$-Calabi--Yau} object, if $E[d]$ is a Serre dual of $E$,
  \item \emph{$d$-spherelike} if $\Hom^*(E,E) = \IC\oplus \IC[-d]$, and
  \item \emph{$d$-spherical} if $E$ is $d$-spherelike and a $d$-Calabi--Yau object.
\end{itemize}
Note a smooth compact variety $X$ of dimension $d$ is a strict Calabi--Yau variety precisely if the structure sheaf $\reg_X$ is a $d$-spherical object of $\Db(X)$ .
 
In \cite{HKP} the authors show that if $E$ is a $d$-spherelike object, there exists a unique maximal triangulated subcategory of $\cT$ in which $E$ becomes $d$-spherical. In the following we will imitate this construction for a larger class of objects.
 
\begin{definition}
Let $E\in \cT$ be an object in a triangulated category having a Serre dual $\mathsf S E$. We call $E$ a \textit{$d$-selfdual object} if
\begin{enumerate}
\item
$\Hom(E, E[d]) \cong \IC$, i.e.\ by Serre duality there is a morphism $w\colon E\to \omega(E) \coloneqq \Serre E[-d]$ unique up to scalars, and
\item
the induced map $w_*\colon \Hom^*(E,E)\isomor \Hom^*(E, \omega(E))$ is an isomorphism.
\end{enumerate}
In particular, a $d$-selfdual object satisfies $\Hom^*(E,E)\cong \Hom(E,E)^\vee[-d]$, hence the name.
\end{definition}
 
\begin{remark}
If an object is $d$-spherelike, then it is $d$-selfdual; compare \cite[Lem.~4.2]{HKP}.
\end{remark}
 
For a $d$-selfdual object $E$, there is a triangle $E \xrightarrow{w} \omega(E)\to Q_E\to E[1]$ induced by $w$. By our assumption, we get $\Hom^*(E,Q_E)=0$. Thus, following an idea suggested by Martin Kalck after discussing \cite[\S7]{HKP} with Michael Wemyss, we propose the following
 
\begin{definition}\label{def:CYN}
The \emph{Calabi-Yau neighbourhood} of a $d$-selfdual object $E\in\cT$ is the full triangulated subcategory
\[
   \CY(E) \coloneqq \lorth Q_E \subseteq \cT  \,.
\]
\end{definition}
 
\begin{proposition} \label{prop:CY-neighbourhood}
If $E\in\cT$ is a $d$-selfdual object then $E\in\CY(E)$ is a $d$-Calabi-Yau object.
\end{proposition}
 
\begin{proof}
If $T\in \CY(E)$, apply $\Hom^*(T,-)$ to the triangle $E\to \omega(E)\to Q_E$.
\end{proof}
 
Using the same proof as for \cite[Thm.~4.6]{HKP}, we see that the Calabi-Yau neighbourhood is the maximal subcategory of $\cT$ in which a $d$-selfdual object $E$ becomes $d$-Calabi-Yau.
 
\begin{proposition} \label{prop:maximality-of-CY-neighbourhood}
If $\ku\subset \cT$ is a full triangulated subcategory and $E\in \ku$ is $d$-Calabi-Yau, then $\ku\subset \CY(E)$. 
\end{proposition}
 
\begin{proposition}\label{prop:WC=CY}
Let $Y$ be a projective variety with rational Gorenstein singularities and trivial canonical bundle of dimension $d=\dim Y$ and consider a resolution of singularities $\rho\colon \wY \to Y$. Then, for every line bundle $L\in \Pic Y$, the pull-back
 $\rho^*L\in \Db(\wY)$ is $d$-selfdual. Furthermore, we have 
\begin{align}\label{WC=CY}
   \WC(\rho) = \bigcap_{L\in \Pic Y} \CY(\rho^*L)  \,. 
\end{align}
\end{proposition}

\begin{proof}
Note that, by our assumption that $\omega_Y$ is trivial, we have $\omega_\wwY\cong \omega_\rho$. 
Hence, by Grothendieck duality, there is a morphism $w_L\colon \rho^*L\to \rho^*L\otimes \omega_\wwY$ unique up to scalar multiplication, namely $w_L=\id_{\rho^*L}\otimes s$ where $s$ is the non-zero section of $\omega_\wwY\cong \omega_\rho$; compare the previous \autoref{sub:catresdef}. Furthermore, $w_{L*} \colon \Hom^*(\rho^*L,\rho^*L) \to \Hom^*(\rho^*L,\rho^*L\otimes \omega_\wwY)$
is an isomorphism, still by Grothendieck duality, which means that $\rho^*L$ is $d$-selfdual. 

Recall that $\WC(\rho)=\ker(\rho_+)$ where $\rho_+$ is defined as the cone
\[
   \rho_* \xrightarrow{t} \rho_! \to \rho_+ \to  \,.
\]
By adjunction, we get $\WC(\rho)=\lorth(\rho^+(\Dperf(Y)))$ where $\rho^+=\rho_+^R$ is given by the triangle
\[
   \rho^+ \to \rho^* \xrightarrow{t^R} \rho^! \to  \,.
\]
Note that $t^R=(\_)\otimes s$. Hence $t^R(L)=w_L\colon \rho^*L\to \rho^*L\otimes \omega_\wwY$ and $\rho^+(L)=Q_{\rho^*L}[-1]$; compare \autoref{def:CYN}. Since the line bundles form a generator of $\Dperf(Y)$, we get for $F\in \Db(\wY)$:
\begin{align*}
         F\in \WC(\rho)
  &\iff  F\in \lorth(\rho^+(\Dperf(Y))) \\
  &\iff  F\in \lorth Q_{\rho^*L} \quad \forall L\in \Pic Y \\
  &\iff  F\in \CY(\rho^*L) \quad \forall L\in \Pic Y  \,. \qedhere
\end{align*}
\end{proof}

\begin{remark}
Following the proof of \autoref{prop:WC=CY}, we see that, on the right-hand side of \eqref{WC=CY}, it is sufficient to take the intersection over all powers of a given ample line bundle.

In our cyclic setup, if $S$ consists of isolated points, we even have $\WC(\rho)=\CY(\reg_{\wwY})$ so that the weakly crepant neighbourhood is computed by a Calabi-Yau neighbourhood of a single object. The same should hold in general if $Y$ has isolated singularities. 
\end{remark}

\section{Stability conditions for Kummer threefolds} \label{sec:Kummer}

\noindent
Let $A$ be an abelian variety of dimension $g$. Consider the action of $G=\mu_2$ by $\pm1$. Then the fixed point set $A[2]$ consists of the $4^g$ two-torsion points. Consider the quotient $\overline{A}$ (the \emph{singular Kummer variety}) of $A$ by $G$, and the blow-up $K(A)$ (the \emph{Kummer resolution}) of $\overline{A}$ in $A[2]$. This setup satisfies \autoref{cond:main}, with $m=2$ and $n=g$ and we get

\begin{corollary}\label{cor:kummersod}
The functor $\Psi\colon  \Db_G(A)\to \Db(K(A))$ is fully faithful, and
\[
   \Db(K(A))
 = \Sod{ \underbrace{\Db(\pt),\dots,\Db(\pt)}_{(g-2)4^{g} \text{ times}}, \Psi(\Db_{G}(A)) }  \,.
\]
\end{corollary}

To explore a potentially useful consequence of this result, we need to recall that a Bridgeland stability condition on a reasonable $\IC$-linear triangulated category $\kd$ consists of the heart $\ka$ of a bounded t-structure in $\kd$ and a function from the numerical Grothendieck group of $\kd$ to the complex numbers satisfying some axioms, see \cite{Bri-stabcond}.

\begin{corollary} \label{cor:stability}
There exists a Bridgeland stability condition on $\Db(K(A))$, for an abelian threefold $A$.
\end{corollary}

\begin{proof}
To begin with, by \cite[Cor.~10.3]{BMS-ab} there is a stability condition on $\Db_G(A)$. Denote by $\ka\subset\Db_G(A)$ the corresponding heart; it is a tilt of the standard heart \cite[\S2]{BMS-ab}.
  
For a two-torsion point $x\in A[2]$, we set $E_x \coloneqq \reg_{\rho^{-1}(\pi(x))}(-1)$. Then, since $g=\dim A=3$, the semi-orthogonal decomposition of \autoref{cor:kummersod} is given by
\begin{align} \label{eq:Kummersod}
  \Db(K(A)) = \Sod{ \{E_x\}_{x\in A[2]}, \Db_G(A) }  \,.
\end{align}

Next, we want to show that, for every $x\in A[2]$, there exists an integer $i$ such that $\Hom^{\leq i}(E_{x},\Psi(F))=0$ for all $F \in \cA \subset \Db_G(A)$. Indeed, the cohomology of any complex in the heart of the stability condition on $\Db_G(A)$, as constructed in \cite[Cor.~10.3]{BMS-ab}, is concentrated in an interval of length three. The functor $\Psi$ has cohomological amplitude at most $3$, since
  $q_*^G \colon \Coh_G(A) \to \Coh(K(A))$
is an exact functor of abelian categories, and every sheaf on $A$ has a locally free resolution of length $\dim A=3$. This implies that the cohomology of any complex in $\Psi(\Db_{G}(A)))$ is contained in a fixed interval of length $6$. This proves the above claim. 
Using \cite[Prop.~3.5(b)]{ColPol-gluestab}, this then implies that $\sod{ E_{x},\Psi(\Db_{G}(A)) }$ has a stability condition; compare the argument in \cite[Cor.~3.8]{BMMS}. 

We can proceed to show that, for $x\neq y\in A[2]$, there exists an integer $i$ such that
  $\Hom^{\leq i}(E_{y},\sod{ E_{x},\Psi(\Db_{G}(A)) }) = 0$
and so there is a stability condition on $\sod{E_y, E_{x},\Psi(\Db_{G}(A))}$. After $4^3$ steps we have constructed a stability condition on $\Db(K(A))$; compare \eqref{eq:Kummersod}.
\end{proof}

\bibliographystyle{alpha}
\bibliography{references}

\bigskip\bigskip\noindent\scriptsize
Contact:

\medskip\noindent
\begin{tabular}{@{} p{0.07\textwidth} @{} p{0.93\textwidth} @{}}
A. K.: & Philipps-Universit\"at Marburg,
         Hans-Meerwein-Stra{\ss}e 6, Campus Lahnberge,
         35032 Marburg, Germany \\
       & Email: \texttt{andkrug@mathematik.uni-marburg.de} \\
D. P.: & Freie Universit\"at Berlin,
         Mathematisches Institut,
         Arnimallee 3,
         14195 Berlin, Germany \\
       & Email \texttt{dploog@math.fu-berlin.de} \\
P. S.: & Universit\"at Hamburg,
         Fachbereich Mathematik,
         Bundesstra{\ss}e 55,
         20146 Hamburg, Germany \\
       & Email \texttt{pawel.sosna@math.uni-hamburg.de}
\end{tabular}

\vfill

\small
\noindent
\phantomsection\label{notation_recap}
\parbox{0.5\textwidth}{
\center
$
%
\xymatrix@C=2.0em@R=5ex{%
       \wX \ar@{->>}^p_{\text{blow-up in $S$}}[rrr] \ar@{->>}_q[dd]
 & & & X   \ar@{->>}^\pi[dd] \\
   & Z=\IP(N) \ar_j[ul] \ar^i[dl] \ar@{->>}^-\nu[r] & S\ar^a[ur] \ar_b[dr] \\
       \wY \save[]+<-2.3em,0.1em>*{\wX/G=}\restore \ar@{->>}_\rho^{\text{blow-up in $S$}}[rrr]
 & & & Y   \save[]+<2.2em,0em>*{=X/G}\restore &
}
$

\bigskip

\center
$\xymatrix@C=7em{
  \Db(\wY) \ar@<2.5pt>[r]^\Phi & \Db_G(X) \ar@<2.5pt>[l]^\Psi \\
  \Db(S) \ar[u]^-{\Xi_\alpha} & \Db(S) \ar[u]_-{\Theta_\beta}
}$
}
\hfill
\parbox{0.4\textwidth}{
$G = \mu_m = \sod{g} \text{ acts on smooth } X \\
 S = \Fix(G) \subset X, n = \dim(X)-\dim(S) \\
 N = N_{S/X} \text{ with } g_{|N} = \zeta \cdot \id_N \\
 \chi\colon G\to\IC^*, \chi(g) = \zeta^{-1} \\
 \\
 \cL_\wY\in\Pic(\wY) \text{ with } \cL_\wX^m = \reg_\wY(Z) \\
 \cL_\wX = \reg_\wX(Z) \in\Pic^G(\wX) \text{ with trivial} \\
 \text{action on } {\cL_\wX}_{|Z} = \reg_Z(Z) = \reg_\nu(-1) \\
 \\ 
 \begin{array}{@{} l @{\:\coloneqq\:} l @{}}
 \Phi         & p_*\circ q^*\circ\triv \\
 \Psi         & (-)^G\circ q_*\circ p^* \\
 \Theta_\beta & i_*(\nu^*(\_)\otimes \reg_\nu(\beta)) \\ 
 \Xi_\alpha   & (a_*\circ\triv)\otimes \chi^\alpha 
 \end{array}$
}
\end{document}